\documentclass[11pt]{article}
\usepackage{amsmath,amssymb,amsthm,amscd}

\allowdisplaybreaks[1]

\newtheorem{theorem}{Theorem}[section]
\newtheorem{lemma}[theorem]{Lemma}
\newtheorem{proposition}[theorem]{Proposition}
\newtheorem{corollary}[theorem]{Corollary}
\newtheorem{conjecture}[theorem]{Conjecture}

\theoremstyle{definition}
\newtheorem{definition}[theorem]{Definition}
\newtheorem{remark}[theorem]{Remark}

\numberwithin{equation}{section}

\renewcommand{\phi}{\varphi}

\newcommand{\ep}{\varepsilon}

\newcommand{\A}{\operatorname{A}}
\renewcommand{\S}{\operatorname{S}}

\newcommand{\Aff}{\operatorname{Aff}}
\newcommand{\BF}{\operatorname{BF}}
\newcommand{\Coker}{\operatorname{Coker}}
\newcommand{\Ext}{\operatorname{Ext}}
\newcommand{\Homeo}{\operatorname{Homeo}}
\newcommand{\id}{\operatorname{id}}
\newcommand{\Ima}{\operatorname{Im}}
\newcommand{\Inf}{\operatorname{Inf}}
\newcommand{\Ker}{\operatorname{Ker}}
\newcommand{\supp}{\operatorname{supp}}
\newcommand{\sgn}{\operatorname{sgn}}
\newcommand{\Tor}{\operatorname{Tor}}
\newcommand{\Tot}{\operatorname{Tot}}

\newcommand{\ab}{\mathrm{ab}}
\newcommand{\K}{\mathcal{K}}
\newcommand{\G}{\mathcal{G}}
\renewcommand{\H}{\mathcal{H}}

\newcommand{\N}{\mathbb{N}}
\newcommand{\Z}{\mathbb{Z}}
\newcommand{\Q}{\mathbb{Q}}
\newcommand{\R}{\mathbb{R}}

\title{\'Etale groupoids arising from \\
products of shifts of finite type}
\author{Hiroki Matui \\
Graduate School of Science \\
Chiba University \\
Inage-ku, Chiba 263-8522, Japan}
\date{}

\begin{document}
\maketitle

\begin{abstract}
Two conjectures about homology groups, $K$-groups 
and topological full groups of minimal \'etale groupoids on Cantor sets 
are formulated. 
We verify these conjectures for many examples of \'etale groupoids 
including products of \'etale groupoids 
arising from one-sided shifts of finite type. 
Furthermore, 
we completely determine when these product groupoids are mutually isomorphic. 
Also, the abelianization of their topological full groups are computed. 
They are viewed as generalizations of the higher dimensional Thompson groups. 
\end{abstract}

%%%%%%%%%%%%%%%%%%%%%%%%%%%%%%%%%%%%%%%%%%%%%%%%%%%%%%%%%%%%
\section{Introduction}

One can construct \'etale groupoids 
from various topological dynamical systems on Cantor sets. 
Minimal $\Z$-actions (i.e. minimal homeomorphisms) provide 
the most important example of dynamical systems on Cantor sets. 
The study of these dynamics was initiated 
by T. Giordano, I. F. Putnam and C. F. Skau \cite{GPS95crelle}, 
in which they classified minimal $\Z$-actions up to orbit equivalence. 
The associated \'etale groupoids and their topological full groups 
have been also studied extensively (\cite{GPS99Israel,M06IJM,JM13Ann}). 
Another fundamental example of dynamical systems on Cantor sets is 
the class of shifts of finite type (also called topological Markov shifts). 
These dynamical systems give 
symbolic representations of hyperbolic dynamical systems 
through Markov partitions. 
The \'etale groupoids of one-sided shifts of finite type (SFT) and 
the associated topological full groups were studied in \cite{M15crelle}. 

With an \'etale groupoid $\G$, 
we can associate the reduced groupoid $C^*$-algebra $C^*_r(\G)$. 
This construction has played 
an important role in the theory of $C^*$-algebras. 
When $\G$ is a groupoid of a minimal $\Z$-action on a Cantor set, 
the $C^*$-algebra $C^*_r(\G)$ is known to be an AT algebra with real rank zero 
(\cite{P90ETDS,GPS95crelle}). 
When $\G$ is a groupoid of a one-sided SFT, 
the $C^*$-algebra $C^*_r(\G)$ is the so-called Cuntz-Krieger $C^*$-algebra 
(\cite{CK80Invent}). 
In general, when $\G$ is minimal and essentially principal, 
$C^*_r(\G)$ is a simple $C^*$-algebra. 
It is a central problem to classify simple $C^*$-algebras by the $K$-groups. 
The study of the $K$-groups $K_i(C^*_r(\G))$ of the groupoid $C^*$-algebra 
has importance from this standpoint. 

In the present paper, we discuss relationship 
between the $K$-groups $K_i(C^*_r(\G)))$ and the homology groups $H_n(\G)$. 
More precisely, we conjecture that 
$\bigoplus_nH_{2n+i}(\G)$ is isomorphic to $K_i(C^*_r(\G)))$ 
for any minimal essentially principal \'etale groupoid $\G$ 
whose unit space is a Cantor set (Conjecture \ref{HK}), 
and verify it for many examples. 
Especially, the conjecture is true 
for minimal $\Z$-actions, one-sided SFT and any products of them 
(Theorem \ref{HKforproduct}). 

The other conjecture states that the sequence 
\[
\begin{CD}
H_0(\G)\otimes\Z_2@>>>[[\G]]_\ab@>>>H_1(\G)@>>>0
\end{CD}
\]
is exact for any minimal essentially principal \'etale groupoid $\G$ 
whose unit space is a Cantor set (Conjecture \ref{AH}), 
where $[[\G]]_\ab$ is the abelianization of 
the topological full group $[[\G]]$. 
Topological full groups provide many interesting examples of discrete groups 
from the viewpoint of combinatorial and geometric group theory. 
For minimal $\Z$-actions, 
various properties of the topological full groups were studied 
in \cite{GPS99Israel}. 
Among others, it was shown that 
the isomorphism class of topological full groups is a complete invariant 
for flip conjugacy. 
Later, in \cite{M06IJM}, it was shown that 
the commutator subgroup $D([[\G]])$ is simple 
and that $D([[\G]])$ is finitely generated if and only if 
the $\Z$-action is expansive. 
K. Juschenko and N. Monod \cite{JM13Ann} proved that $[[\G]]$ is amenable. 
By these results, we obtained infinitely many countable infinite groups 
which are simple, finitely generated and amenable. 
When $\G$ is a groupoid of one-sided SFT, in \cite{M15crelle}, 
it was proved that $[[\G]]$ is of type F$_\infty$ 
(in particular, finitely presented) 
and that $D([[\G]])$ is simple. 
They are regarded as generalizations of the Higman-Thompson groups 
(see \cite{Ne04JOP,M15crelle}). 
Also, K. Matsumoto and the author \cite{MM14Kyoto} completely determined 
when these $[[\G]]$ are mutually isomorphic. 
In this paper, we will prove that 
Conjecture \ref{AH} holds for many \'etale groupoids, 
which include almost finite groupoids (Theorem \ref{almstfin>AH}) 
and purely infinite groupoids with property TR (Theorem \ref{pi>AH}). 

The latter half of the paper is devoted to 
the study of product groupoids $\G$ of finitely many SFT groupoids. 
Topological full groups $[[\G]]$ of these groupoids are thought of 
as generalizations of the higher dimensional Thompson groups $nV_{k,r}$ 
(see Remark \ref{highdimThomp}). 
First, we find a necessary and sufficient condition 
so that two such groupoids become isomorphic to each other 
(Theorem \ref{prodSFTclassify}). 
This is an extension of the classification result of 
W. Dicks and C. Mart\'inez-P\'erez \cite{DMP14Israel}. 
Furthermore, 
we compute the abelianization $[[\G]]_\ab$ of the topological full group 
(Theorem \ref{prodSFTAbel}). 
As a special case, the abelianization of $nV_{k,r}$ is given explicitly 
(Theorem \ref{highdimThompAbel}).

%%%%%%%%%%%%%%%%%%%%%%%%%%%%%%%%%%%%%%%%%%%%%%%%%%%%%%%%%%%%
\section{Preliminaries}

%%%%%%%%%%%%%%%%%%%%%%%%%%%%%%%%%%%%%%%%%%%%%%%%%%%%%%%%%%%%
\subsection{\'Etale groupoids}

The cardinality of a set $A$ is written $\#A$ and 
the characteristic function of $A$ is written $1_A$. 
The finite cyclic group of order $n$ is denoted 
by $\Z_n=\{\bar{r}\mid r=1,2,\dots,n\}$. 
We say that a subset of a topological space is clopen 
if it is both closed and open. 
A topological space is said to be totally disconnected 
if its topology is generated by clopen subsets. 
By a Cantor set, 
we mean a compact, metrizable, totally disconnected space 
with no isolated points. 
It is known that any two such spaces are homeomorphic. 
The homeomorphism group of a topological space $X$ is written $\Homeo(X)$. 
The commutator subgroup of a group $\Gamma$ is denoted by $D(\Gamma)$. 
We let $\Gamma_\ab$ denote the abelianization $\Gamma/D(\Gamma)$. 

In this article, by an \'etale groupoid 
we mean a second countable locally compact Hausdorff groupoid 
such that the range map is a local homeomorphism. 
We refer the reader to \cite{Rtext,R08Irish} 
for background material on \'etale groupoids. 
For an \'etale groupoid $\G$, 
we let $\G^{(0)}$ denote the unit space and 
let $s$ and $r$ denote the source and range maps. 
For $x\in\G^{(0)}$, 
$\G(x)=r(\G x)$ is called the $\G$-orbit of $x$. 
When every $\G$-orbit is dense in $\G^{(0)}$, 
$\G$ is said to be minimal. 
For a subset $Y\subset\G^{(0)}$, 
the reduction of $\G$ to $Y$ is $r^{-1}(Y)\cap s^{-1}(Y)$ and 
denoted by $\G|Y$. 
If $Y$ is clopen, then 
the reduction $\G|Y$ is an \'etale subgroupoid of $\G$ in an obvious way. 
For $x\in\G^{(0)}$, 
we write $\G_x=r^{-1}(x)\cap s^{-1}(x)$ and call it the isotropy group of $x$. 
The isotropy bundle of $\G$ is 
$\G'=\{g\in\G\mid r(g)=s(g)\}=\bigcup_{x\in\G^{(0)}}\G_x$. 
We say that $\G$ is principal if $\G'=\G^{(0)}$. 
When the interior of $\G'$ is $\G^{(0)}$, 
we say that $\G$ is essentially principal. 
The set of interior points $(\G')^\circ$ of $\G'$ forms 
an \'etale subgroupoid. 
The quotient of $\G$ by the equivalence relation 
\[
g\sim h\iff gh^{-1}\in(\G')^\circ
\]
becomes an essentially principal \'etale groupoid. 
We call it the essentially principal part of $\G$. 

A subset $U\subset\G$ is called a $\G$-set if $r|U,s|U$ are injective. 
Any open $\G$-set $U$ induces the homeomorphism 
$(r|U)\circ(s|U)^{-1}$ from $s(U)$ to $r(U)$. 
We write $\theta(U)=(r|U)\circ(s|U)^{-1}$. 
When $U,V$ are $\G$-sets, 
\[
U^{-1}=\{g\in\G\mid g^{-1}\in U\}
\]
and 
\[
UV=\{gg'\in\G\mid g\in U,\ g'\in V,\ s(g)=r(g')\}
\]
are also $\G$-sets. 
A probability measure $\mu$ on $\G^{(0)}$ is said to be $\G$-invariant 
if $\mu(r(U))=\mu(s(U))$ holds for every open $\G$-set $U$. 
The set of all $\G$-invariant probability measures is denoted by $M(\G)$. 

Let $\phi:\Gamma\curvearrowright X$ be 
an action of a countable discrete group $\Gamma$ 
on a locally compact Hausdorff space $X$ by homeomorphisms. 
We let $\G_\phi=\Gamma\times X$ and define the following groupoid structure: 
$(\gamma,x)$ and $(\gamma',x')$ are composable 
if and only if $x=\phi^{\gamma'}(x')$, 
$(\gamma,\phi^{\gamma'}(x'))\cdot(\gamma',x')=(\gamma\gamma',x')$ and 
$(\gamma,x)^{-1}=(\gamma^{-1},\phi^\gamma(x))$. 
Then $\G_\phi$ is an \'etale groupoid and 
called the transformation groupoid 
arising from $\phi:\Gamma\curvearrowright X$. 

We would like to recall 
the notion of topological full groups for \'etale groupoids. 

\begin{definition}[{\cite[Definition 2.3]{M12PLMS}}]\label{defoftfg}
Let $\G$ be an essentially principal \'etale groupoid 
whose unit space $\G^{(0)}$ is a Cantor set. 
The set of all $\alpha\in\Homeo(\G^{(0)})$ for which 
there exists a compact open $\G$-set $U$ 
satisfying $\alpha=\theta(U)$ 
is called the topological full group of $\G$ and denoted by $[[\G]]$. 
\end{definition}

For $\alpha\in[[\G]]$ the compact open $\G$-set $U$ as above uniquely exists, 
because $\G$ is essentially principal. 
Obviously $[[\G]]$ is a subgroup of $\Homeo(\G^{(0)})$. 
Since $\G$ is second countable, it has countably many compact open subsets, 
and so $[[\G]]$ is at most countable. 

In \cite[Section 7]{M12PLMS}, 
we introduced the index map $I:[[\G]]\to H_1(\G)$, 
where $H_1(\G)$ is the homology group of $\G$ 
(see Section 2.2 or \cite[Section 3]{M12PLMS}). 
For $\alpha\in[[\G]]$, 
let $U\subset\G$ be the compact open $\G$-set 
such that $\alpha=\theta(U)=(r|U)\circ(s|U)^{-1}$. 
Then the characteristic function $1_U\in C_c(\G,\Z)$ is a cycle, 
and $I(\alpha)$ is the equivalence class of $1_U$ in $H_1(\G)$. 
The index map $I$ is a homomorphism. 
We let $[[\G]]_0$ denote the kernel of the index map $I$. 
Evidently $D([[\G]])$ is contained in $[[\G]]_0$. 

The following theorem says that 
the groups $[[\G]]$, $[[\G]]_0$ and $D([[\G]])$ `remember' 
the \'etale groupoid $\G$. 

\begin{theorem}[{\cite[Theorem 3.10]{M15crelle}}]
For $i=1,2$, let $\G_i$ be an essentially principal \'etale groupoid 
whose unit space is a Cantor set 
and suppose that $\G_i$ is minimal. 
The following conditions are equivalent. 
\begin{enumerate}
\item $\G_1$ and $\G_2$ are isomorphic as \'etale groupoids. 
\item $[[\G_1]]$ and $[[\G_2]]$ are isomorphic as discrete groups. 
\item $[[\G_1]]_0$ and $[[\G_2]]_0$ are isomorphic as discrete groups. 
\item $D([[\G_1]])$ and $D([[\G_2]])$ are isomorphic as discrete groups. 
\end{enumerate}
\end{theorem}

Recently, V. V. Nekrashevych \cite{Ne15} introduced two normal subgroups 
$\A(\G)\subset\S(\G)\subset[[\G]]$. 
Roughly speaking, 
$\S(\G)$ is the subgroup generated by all elements of order two, 
and $\A(\G)$ is the subgroup generated by all elements of order three 
(see \cite{Ne15} for the precise definitions). 
They are analogs of the symmetric and alternating groups. 
He proved that 
the same statement as the theorem above is true for $\A(\G)$ and $\S(\G)$. 
Furthermore, he proved that $\A(\G)$ is simple if $\G$ is minimal, 
and that $\A(\G)$ is finitely generated if $\G$ is expansive. 
When $\G$ is minimal and almost finite (or purely infinite), 
we can verify $\A(\G)=D([[\G]])$. 
When $\G$ is principal and almost finite, 
we can verify $\S(\G)=[[\G]]_0$. 

For an \'etale groupoid $\G$, 
we denote the reduced groupoid $C^*$-algebra of $\G$ by $C^*_r(\G)$ and 
identify $C_0(\G^{(0)})$ with a subalgebra of $C^*_r(\G)$. 
J. Renault's theorem \cite[Theorem 5.9]{R08Irish} tells us that 
two essentially principal \'etale groupoids $\G_1$ and $\G_2$ are isomorphic 
if and only if there exists an isomorphism $\phi:C^*_r(\G_1)\to C^*_r(\G_2)$ 
such that $\phi(C_0(\G_1^{(0)}))=C_0(\G_2^{(0)})$ 
(see also \cite[Theorem 5.1]{M12PLMS}).

%%%%%%%%%%%%%%%%%%%%%%%%%%%%%%%%%%%%%%%%%%%%%%%%%%%%%%%%%%%%
\subsection{Homology groups and the K\"unneth theorem}

The homology groups $H_n(\G)$ of an \'etale groupoid $\G$ were first 
introduced and studied by M. Crainic and I. Moerdijk in \cite{CM00crelle}. 
In the case that the unit space $\G^{(0)}$ is a Cantor set, 
we investigated connections 
between the homology groups and dynamical properties of $\G$ 
in \cite{M12PLMS, M15crelle}. 
Let us recall the definition of $H_n(\G)$. 

Let $A$ be a topological abelian group. 
For a locally compact Hausdorff space $X$, we denote by $C_c(X,A)$ 
the set of $A$-valued continuous functions with compact support. 
When $X$ is compact, we simply write $C(X,A)$. 
With pointwise addition, $C_c(X,A)$ is an abelian group. 
Let $\pi:X\to Y$ be a local homeomorphism 
between locally compact Hausdorff spaces. 
For $f\in C_c(X,A)$, we define a map $\pi_*(f):Y\to A$ by 
\[
\pi_*(f)(y)=\sum_{\pi(x)=y}f(x). 
\]
It is not so hard to see that $\pi_*(f)$ belongs to $C_c(Y,A)$ and 
that $\pi_*$ is a homomorphism from $C_c(X,A)$ to $C_c(Y,A)$. 
Besides, if $\pi':Y\to Z$ is another local homeomorphism to 
a locally compact Hausdorff space $Z$, then 
one can check $(\pi'\circ\pi)_*=\pi'_*\circ\pi_*$ in a direct way. 
Thus, $C_c(\cdot,A)$ is a covariant functor 
from the category of locally compact Hausdorff spaces 
with local homeomorphisms 
to the category of abelian groups with homomorphisms. 

Let $\G$ be an \'etale groupoid. 
For $n\in\N$, we write $\G^{(n)}$ 
for the space of composable strings of $n$ elements in $\G$, that is, 
\[
\G^{(n)}=\{(g_1,g_2,\dots,g_n)\in\G^n\mid
s(g_i)=r(g_{i+1})\text{ for all }i=1,2,\dots,n{-}1\}. 
\]
For $i=0,1,\dots,n$, 
we let $d_i:G^{(n)}\to G^{(n-1)}$ be a map defined by 
\[
d_i(g_1,g_2,\dots,g_n)=\begin{cases}
(g_2,g_3,\dots,g_n) & i=0 \\
(g_1,\dots,g_ig_{i+1},\dots,g_n) & 1\leq i\leq n{-}1 \\
(g_1,g_2,\dots,g_{n-1}) & i=n. 
\end{cases}
\]
When $n=1$, we let $d_0,d_1:\G^{(1)}\to\G^{(0)}$ be 
the source map and the range map, respectively. 
For $j=0,1,\dots,n$, 
we let $s_j:G^{(n)}\to G^{(n+1)}$ be a map defined by 
\[
s_j(g_1,g_2,\dots,g_n)=\begin{cases}
(r(g_1),g_1,g_2,\dots,g_n) & j=0 \\
(g_1,\dots,g_j,r(g_{j+1}),g_{j+1},\dots,g_n) & 1\leq j\leq n{-}1 \\
(g_1,g_2,\dots,g_n,s(g_n)) & j=n. 
\end{cases}
\]
Clearly the maps $d_i$ and $s_j$ are local homeomorphisms. 
The spaces $(\G^{(n)})_{n\geq0}$ 
together with the face maps $d_i$ and the degeneracy maps $s_j$ 
form a simplicial space, 
which is called the nerve of $\G$ 
(see \cite[Example I.1.4]{GJtext} for instance). 
Applying the covariant functor $C_c(\cdot,A)$ to the nerve of $\G$, 
we obtain the simplicial abelian group $(C_c(\G^{(n)},A))_{n\geq0}$. 
Define the homomorphisms $\partial_n:C_c(\G^{(n)},A)\to C_c(\G^{(n-1)},A)$ 
by 
\[
\partial_n=\sum_{i=0}^n(-1)^id_{i*}. 
\]
The abelian groups $C_c(\G^{(n)},A)$ 
together with the boundary operators $\partial_n$ form a chain complex, 
which is called the Moore complex of 
the simplicial abelian group $(C_c(\G^{(n)},A))_n$ 
(see \cite[Chapter III.2]{GJtext} for instance). 

\begin{definition}
[{\cite[Section 3.1]{CM00crelle},\cite[Definition 3.1]{M12PLMS}}]
\label{homology}
We let $H_n(\G,A)$ be the homology groups of the Moore complex above, 
i.e. $H_n(\G,A)=\Ker\partial_n/\Ima\partial_{n+1}$, 
and call them the homology groups of $\G$ with constant coefficients $A$. 
When $A=\Z$, we simply write $H_n(\G)=H_n(\G,\Z)$. 
In addition, we define 
\[
H_0(\G)^+=\{[f]\in H_0(\G)\mid f(x)\geq0\text{ for all }x\in\G^{(0)}\}, 
\]
where $[f]$ denotes the equivalence class of $f\in C_c(\G^{(0)},\Z)$. 
\end{definition}

When $\G=\G_\phi$ is the transformation groupoid 
arising from a group action $\phi:\Gamma\curvearrowright X$ 
on a Cantor set $X$, 
$H_n(\G)$ is naturally isomorphic to the group homology $H_n(\Gamma,C(X,\Z))$. 

For the homology groups $H_n(\G,A)$, the following K\"unneth theorem holds. 

\begin{theorem}\label{Kunneth}
Let $\G$ and $\H$ be \'etale groupoids. 
For any $n\geq0$, there exists a natural short exact sequence 
\[
0\longrightarrow\bigoplus_{i+j=n}H_i(\G)\otimes H_j(\H)
\longrightarrow H_n(\G\times\H)
\longrightarrow\bigoplus_{i+j=n-1}\Tor(H_i(\G),H_j(\H))\longrightarrow0. 
\]
Furthermore these sequences split (but not canonically). 
\end{theorem}
\begin{proof}
The family of abelian groups $C_c(\G^{(m)},\Z)\otimes C_c(\H^{(n)},\Z)$ 
form a bisimplicial abelian group $\mathcal{A}$ 
(see \cite[Chapter IV.2.2]{GJtext} for instance). 
We will apply the generalized Eilenberg-Zilber theorem 
\cite[Theorem IV.2.4]{GJtext} 
to this bisimplicial abelian group $\mathcal{A}$. 
Its diagonal simplicial abelian group $d(\mathcal{A})$ is 
canonically identified 
with the simplicial abelian group $(C_c((\G\times\H)^{(n)},\Z))_n$, 
because 
\[
C_c(\G^{(n)},\Z)\otimes C_c(\H^{(n)},\Z)
\cong C_c(\G^{(n)}\times\H^{(n)},\Z)\cong C_c((\G\times\H)^{(n)},\Z). 
\]
On the other hand, 
we consider the Moore bicomplex 
for the bisimplicial abelian group $\mathcal{A}$. 
Evidently its total complex $\Tot(\mathcal{A})$ is equal to 
the tensor product of the Moore complex of $(C_c(\G^{(m)},\Z))_m$ and 
the Moore complex of $(C_c(\H^{(n)},\Z))_n$. 
It follows from \cite[Theorem IV.2.4]{GJtext} that 
the Moore complex of $d(\mathcal{A})$ is 
chain homotopy equivalent to $\Tot(\mathcal{A})$. 
Hence the chain complex $(C_c((\G\times\H)^{(n)},\Z))_n$ is 
chain homotopy equivalent to 
the tensor product of $(C_c(\G^{(m)},\Z))_m$ and $(C_c(\H^{(n)},\Z))_n$. 
Then, 
the usual K\"unneth formula applies and yields the desired conclusion. 
\end{proof}

\begin{corollary}\label{H1ofprod}
Let $\G=\G_1\times \G_2\times\dots\times\G_m$ be a product groupoid of 
\'etale groupoids $\G_1,\G_2,\dots,\G_m$. 
We have 
\[
H_0(\G)\cong H_0(\G_1)\otimes H_0(\G_2)\otimes\dots\otimes H_0(\G_m)
\]
and 
\begin{align*}
H_1(\G)&\cong\bigoplus_{i_1+i_2+\dots+i_m=1}
H_{i_1}(\G_1)\otimes H_{i_2}(\G_2)\otimes\dots\otimes H_{i_m}(\G_m)\\
&\quad\oplus\bigoplus_{p=1}^{m-1}
H_0(\G_1)\otimes\dots\otimes H_0(\G_{p-1})
\otimes\Tor(H_0(\G_p),H_0(\G_{p+1}\times\dots\times \G_m)). 
\end{align*}
\end{corollary}
\begin{proof}
Repeated use of the K\"unneth theorem yields the conclusions. 
We remark that (1) can be shown directly from the definition 
without appealing to the K\"unneth theorem. 
\end{proof}

%%%%%%%%%%%%%%%%%%%%%%%%%%%%%%%%%%%%%%%%%%%%%%%%%%%%%%%%%%%%
\subsection{Conjectures}

In this subsection, we formulate two conjectures, 
namely the HK conjecture (Conjecture \ref{HK}) and 
the AH conjecture (Conjecture \ref{AH}). 
The first conjecture says that 
the homology groups $H_n(\G)$ agree with 
the $K$-groups of the $C^*$-algebra $C^*_r(\G)$. 

\begin{conjecture}[HK conjecture]\label{HK}
Let $\G$ be an essentially principal minimal \'etale groupoid 
whose unit space $\G^{(0)}$ is a Cantor set. 
Then we have 
\[
\bigoplus_{i=0}^\infty H_{2i}(\G)\cong K_0(C^*_r(\G))
\]
and 
\[
\bigoplus_{i=0}^\infty H_{2i+1}(\G)\cong K_1(C^*_r(\G)). 
\]
\end{conjecture}

Recall that 
two minimal \'etale groupoids $\G_1$ and $\G_2$ are said to be 
Morita equivalent (or Kakutani equivalent) 
if there exists a clopen subset $Y_i\subset\G_i^{(0)}$ for $i=1,2$ 
such that $\G_1|Y_1\cong\G_2|Y_2$ 
(see \cite[Definition 4.1]{M12PLMS} for instance). 

\begin{proposition}
Let $\G$ and $\H$ be essentially principal minimal \'etale groupoids 
whose unit spaces are Cantor sets. 
If $\G$ and $\H$ are Morita equivalent and 
$\G$ satisfies the HK conjecture, then so does $\H$. 
\end{proposition}
\begin{proof}
Let $\G$ be an essentially principal minimal \'etale groupoid 
and let $Y\subset\G^{(0)}$ be a clopen subset. 
Then $C^*_r(\G|Y)$ is canonically isomorphic to 
the hereditary subalgebra $1_YC^*_r(\G)1_Y$ of $C^*_r(\G)$. 
Therefore, if $\G$ and $\H$ are Morita equivalent, then 
$C^*_r(\G)$ is strongly Morita equivalent to $C^*_r(\H)$. 
It is well-known that 
$K$-groups of $C^*$-algebras are invariant under strong Morita equivalence. 
On the other hand, 
by \cite[Corollary 4.6]{CM00crelle} (see also \cite[Theorem 4.8]{M12PLMS}), 
homology groups of \'etale groupoids are invariant under Morita equivalence. 
Hence we get the conclusion. 
\end{proof}

As an immediate consequence of Theorem \ref{Kunneth}, 
we get the following. 

\begin{theorem}\label{HKforproduct}
Let $\G$ and $\H$ be \'etale groupoids 
whose unit spaces are Cantor sets. 
Suppose that $C^*_r(\G)$ is nuclear and satisfies the UCT. 
If the HK conjecture is true for $\G$ and $\H$, 
then it is also true for the product groupoid $\G\times\H$. 
\end{theorem}
\begin{proof}
This follows from Theorem \ref{Kunneth} 
(the K\"unneth theorem for groupoids) and 
\cite[Theorem 23.1.3]{Btext} 
(the K\"unneth theorem for $C^*$-algebras). 
\end{proof}

The second conjecture says that 
the abelianization $[[\G]]_\ab$ has a close connection 
to the homology groups $H_0(\G)$ and $H_1(\G)$. 

\begin{conjecture}[AH conjecture]\label{AH}
Let $\G$ be an essentially principal minimal \'etale groupoid 
whose unit space $\G^{(0)}$ is a Cantor set. 
Then there exists an exact sequence 
\[
\begin{CD}
H_0(\G)\otimes\Z_2@>>>[[\G]]_\ab@>I>>H_1(\G)@>>>0. 
\end{CD}
\]
Especially, if $H_0(\G)$ is $2$-divisible, 
then we have $[[\G]]_\ab\cong H_1(\G)$. 
\end{conjecture}

Indeed, in many examples we can verify that 
there exists a short exact sequence 
\[
\begin{CD}
0@>>>H_0(\G)\otimes\Z_2@>>>[[\G]]_\ab@>I>>H_1(\G)@>>>0. 
\end{CD}
\]
In such a case, 
we say that $\G$ satisfies the strong AH property. 

\begin{remark}
At the beginning of this research, 
I had expected that 
this stronger version of AH might be true in general. 
But V. V. Nekrashevych kindly informed me that 
this is not the case. 
For every locally expanding self-covering map $f:M\to M$ 
of a compact path connected metric space $M$, 
he introduced a finitely presented group $V_f$ (\cite{Ne13}). 
The group $V_f$ can be regarded 
as a topological full group of an \'etale groupoid on a Cantor set. 
The abelianization of $V_f$ is computed in \cite[Section 5.3]{Ne13}, 
and it shows that the map $H_0(\G)\otimes\Z_2\to[[\G]]_\ab$ is 
not always injective (see \cite[Proposition 5.8]{Ne13}). 
The reader may find its detailed description in \cite[Example 7.1]{Ne15}. 
After this communication with him, 
I noticed that 
product groupoids of SFT groupoids do not always have the strong AH property. 
We discuss it in Section 5.5. 
\end{remark}

The homomorphism $H_0(\G)\otimes\Z_2\to[[\G]]_\ab$ 
appearing in the AH conjecture is described as follows. 
Let $f\in C(\G^{(0)},\Z)$ and 
let $A=\{x\in\G^{(0)}\mid f(x)\notin2\Z\}$. 
Then $f-1_A$ is in $2C(\G^{(0)},\Z)$. 
Since $\G$ is minimal, for any $x\in A$, 
there exists a compact open $\G$-set $U$ 
such that $x\in s(U)\subset A$ and $r(U)\cap s(U)=\emptyset$. 
Using the compactness of $A$, 
we can find compact open $\G$-sets $U_1,U_2,\dots,U_n$ 
such that $\{s(U_i)\mid i=1,2,\dots,n\}$ is a clopen partition of $A$ 
and $r(U_i)\cap s(U_i)=\emptyset$. 
Define $\tau_i\in[[\G]]$ by 
\[
\tau_i(x)=\begin{cases}\theta(U_i)(x)&x\in s(U_i)\\
\theta(U_i^{-1})(x)&x\in r(U_i)\\x&\text{otherwise. }\end{cases}
\]
The homomorphism $H_0(\G)\otimes\Z_2\to[[\G]]_\ab$ is defined 
by sending the equivalence class of $f$ to 
the equivalence class of $\tau_1\tau_2\dots\tau_n$. 
Notice that it is not clear at all if this is well-defined. 
One can find a proof of the well-definedness in \cite{Ne15}. 

We introduce two properties of $\G$ 
which are related to the AH conjecture. 
We call $\tau\in[[\G]]$ a transposition 
if $\tau^2=1$ and $\{x\in\G^{(0)}\mid\tau(x)=x\}$ is clopen. 
By \cite[Lemma 7.7 (3)]{M12PLMS}, 
any transposition belongs to $[[\G]]_0$. 

\begin{definition}
Let $\G$ be an essentially principal \'etale groupoid 
whose unit space $\G^{(0)}$ is a Cantor set. 
\begin{enumerate}
\item We say that $\G$ has cancellation 
if for any clopen sets $A_1,A_2\subset\G^{(0)}$ 
with $[1_{A_1}]=[1_{A_2}]$ in $H_0(\G)$ 
there exists a compact open $\G$-set $U\subset\G$ 
such that $s(U)=A_1$ and $r(U)=A_2$. 
\item We say that $\G$ has property TR 
if the group $[[\G]]_0$ is generated by transpositions. 
\end{enumerate}
\end{definition}

When $\G$ has cancellation, 
it is easy to see that for any $A_1,A_2$ as above 
we can find a transposition $\tau$ such that $\tau(A_1)=A_2$ 
and $\tau(x)=x$ for $x\in\G^{(0)}\setminus(A_1\cup A_2)$.

%%%%%%%%%%%%%%%%%%%%%%%%%%%%%%%%%%%%%%%%%%%%%%%%%%%%%%%%%%%%
\section{Almost finite groupoids}

In this section we show that Conjecture \ref{AH} holds 
for principal, minimal, almost finite groupoids. 

Let us recall the notion of AF (approximately finite) groupoids 
(\cite[Definition III.1.1]{Rtext}, \cite[Definition 3.7]{GPS04ETDS}). 

\begin{definition}
Let $\G$ be an \'etale groupoid 
whose unit space is compact and totally disconnected. 
\begin{enumerate}
\item We say that $\G$ is elementary if $\G$ is compact and principal. 
\item We say that $\G$ is an AF groupoid 
if there exists an increasing sequence $(\K_n)_n$ of 
elementary subgroupoids of $\G$ 
such that $\K_n^{(0)}=\G^{(0)}$ and $\bigcup_n\K_n=\G$. 
\end{enumerate}
\end{definition}

Let $\G$ be a (not necessarily minimal) AF groupoid. 
The topological full group $[[\G]]$ is written 
as an increasing union of finite direct sums of symmetric groups 
(\cite[Proposition 3.3]{M06IJM}). 
In particular, $[[\G]]$ is locally finite 
(actually, its converse also holds, i.e. 
if $[[\G]]$ is locally finite, then $\G$ is AF, 
see \cite[Proposition 3.2]{M06IJM}). 
The $C^*$-algebra $C^*_r(\G)$ is an AF algebra. 
By \cite[Theorem 4.10, 4.11]{M12PLMS}, 
$H_n(\G)$ is trivial for $n\geq1$ and 
$H_0(\G)$ is isomorphic to $K_0(C^*_r(\G))$ (as a unital ordered group). 
Hence, the HK conjecture holds for $\G$. 
When every $\G$-orbit contains more than one point, 
by \cite[Lemma 3.5]{M06IJM}, 
$[[\G]]_\ab$ is isomorphic to $H_0(\G)\otimes\Z_2$. 
Thus, $\G$ has the strong AH property. 
Moreover, when $\G$ is minimal, $D([[\G]])$ is simple 
(\cite[Lemma 3.5]{M06IJM}). 

Next, we would like to recall the definition of almost finite groupoids. 

\begin{definition}[{\cite[Definition 6.2]{M12PLMS}}]
Let $\G$ be an \'etale groupoid whose unit space is a Cantor set. 
We say that $\G$ is almost finite 
if for any compact subset $C\subset\G$ and $\ep>0$ 
there exists an elementary subgroupoid $\K\subset\G$ such that 
\[
\frac{\#(C\K x\setminus\K x)}{\#(\K x)}<\ep
\]
for all $x\in\G^{(0)}$. 
We also remark that $\#(\K(x))$ equals $\#(\K x)$, because $\K$ is principal. 
\end{definition}

The following are known for almost finite groupoids. 

\begin{theorem}\label{almstfin}
Let $\G$ be an almost finite groupoid. 
\begin{enumerate}
\item If $\G$ is minimal, then $D([[\G]])$ is simple. 
\item The index map $I:[[\G]]\to H_1(\G)$ is surjective. 
\item If $\G$ is minimal, then $\G$ has cancellation. 
\item If $\G$ is principal, then $\G$ satisfies property TR. 
\end{enumerate}
\end{theorem}
\begin{proof}
(1) is \cite[Theorem 4.7]{M15crelle}. 
(2) is \cite[Theorem 7.5]{M12PLMS}. 
(3) is \cite[Theorem 6.12]{M12PLMS}. 

(4)
By \cite[Theorem 7.13]{M12PLMS}, 
any element of $[[\G]]_0$ is a product of four elements of finite order. 
Since $\G$ is principal, any element of finite order is elementary 
in the sense of \cite[Definition 7.6 (1)]{M12PLMS}. 
Therefore it is a product of transpositions. 
\end{proof}

For a $\G$-invariant probability measure $\mu\in M(\G)$, 
we can define a homomorphism $\hat\mu:H_0(\G)\to\R$ by 
\[
\hat\mu([f])=\int f\,d\mu. 
\]
It is clear that 
$\hat\mu([1_{\G^{(0)}}])=1$ and $\hat\mu(H_0(\G)^+)\subset[0,\infty)$. 
Thus $\hat\mu$ is a state on $(H_0(\G),H_0(\G)^+,[1_{\G^{(0)}}])$ 
(see \cite[Definition 6.8.1]{Btext} for instance). 
It is also easy to see that the map $\mu\mapsto\hat\mu$ gives 
an isomorphism from $M(\G)$ to the state space. 

\begin{theorem}\label{almstfin2}
Let $\G$ be a minimal almost finite groupoid. 
\begin{enumerate}
\item $(H_0(\G),H_0(\G)^+)$ is a simple, weakly unperforated, 
ordered abelian group with the Riesz interpolation property. 
\item The homomorphism $\rho:H_0(\G)\to\Aff(M(\G))$ 
defined by $\rho([f])(\mu)=\hat\mu([f])$ has uniformly dense range, 
where $\Aff(M(\G))$ denotes 
the space of $\R$-valued affine continuous functions on $M(\G)$. 
\end{enumerate}
\end{theorem}
\begin{proof}
(1) is \cite[Proposition 6.10]{M12PLMS}. 
By (1), $H_0(\G)/\Tor(H_0(\G))$ is a simple dimension group 
(see \cite[Section 7.4]{Btext} for instance). 
Besides, $H_0(\G)/\Tor(H_0(\G))$ is not equal to $\Z$. 
It follows that $\rho$ has uniformly dense range in $\Aff(M(\G))$. 
\end{proof}

By using these properties of almost finite groupoids, 
we can show the AH conjecture (Theorem \ref{almstfin>AH}). 

\begin{lemma}\label{H_0mod2}
Let $\G$ be a minimal almost finite groupoid. 
Let $f\in C(\G^{(0)},\Z)$. 
For any non-empty clopen set $A\subset\G^{(0)}$, 
there exists a clopen subset $B\subset A$ 
such that $f$ and $1_B$ have the same equivalence class 
in $H_0(\G)\otimes\Z_2$. 
\end{lemma}
\begin{proof}
There exists $c>0$ such that $\mu(A)\geq c$ for any $\mu\in M(\G)$. 
By Theorem \ref{almstfin2} (2), 
there exists $g\in C(\G^{(0)},\Z)$ such that 
\[
\hat\mu(f)-c<2\hat\mu(g)<\hat\mu(f)
\]
for all $\mu\in M(\G)$. 
It follows that $0<\hat\mu(f-2g)<c$ holds for all $\mu\in M(\G)$. 
Hence, by \cite[Lemma 6.7]{M12PLMS}, 
we can find a clopen subset $B\subset A$ such that 
$[1_B]=[f-2g]$ in $H_0(\G)$, 
which means that $1_B$ and $f$ have the same equivalence class 
in $H_0(\G)\otimes\Z_2$. 
\end{proof}

\begin{theorem}\label{almstfin>AH}
Let $\G$ be a principal, minimal, almost finite groupoid. 
Then, Conjecture \ref{AH} is true for $\G$. 
\end{theorem}
\begin{proof}
By Theorem \ref{almstfin} (2), 
the map $[[\G]]_\ab\to H_1(\G)$ is surjective. 

For a compact open $\G$-set $U$ satisfying $r(U)\cap s(U)=\emptyset$, 
we define $\tau_U\in[[\G]]$ by 
\[
\tau_U(x)=\begin{cases}\theta(U)(x)&x\in s(U)\\
\theta(U^{-1})(x)&x\in r(U)\\x&\text{otherwise. }\end{cases}
\]
First, we claim that 
the equivalence class $[\tau_U]$ of $\tau_U$ in $[[\G]]_\ab$ 
depends only on the equivalence class $[1_{s(U)}]$ 
of $1_{s(U)}$ in $H_0(\G)\otimes\Z_2$. 
Suppose that $[1_{s(U)}]$ is zero in $H_0(\G)\otimes\Z_2$. 
By the results of \cite[Section 6.2]{M12PLMS}, 
we can find a clopen subset $U_1\subset U$ 
such that $[1_{s(U)}]=2[1_{s(U_1)}]$ in $H_0(\G)$. 
Let $U_2=U\setminus U_1$. 
By Theorem \ref{almstfin} (3), there exists a compact open $\G$-set $V$ 
such that $s(V)=s(U_1)$ and $r(V)=s(U_2)$. 
Put $W=V\cup U_2VU_1^{-1}$. 
Then we have 
\[
\tau_U=\tau_{U_1}\tau_{U_2}=\tau_{U_1}\tau_W\tau_{U_1}\tau_W\in D([[\G]]), 
\]
and so $[\tau_U]$ is trivial in $[[\G]]_\ab$. 
Suppose that two $\G$-sets $U_1$ and $U_2$ satisfy 
$r(U_i)\cap s(U_i)=\emptyset$ and 
$[1_{s(U_1)}]=[1_{s(U_2)}]$ in $H_0(\G)\otimes\Z_2$. 
We would like to prove that 
$\tau_{U_1}$ and $\tau_{U_2}$ have the same equivalence class in $[[\G]]_\ab$. 
In view of Lemma \ref{H_0mod2}, we can find clopen $\G$-sets $V_i\subset U_i$ 
such that $[1_{s(U_i)}]=[1_{s(V_i)}]$ in $H_0(\G)\otimes\Z_2$ and 
$[1_{s(V_1)}]=[1_{s(V_2)}]$ in $H_0(\G)$. 
Since the equivalence class of 
$1_{s(U_i\setminus V_i)}=1_{s(U_i)\setminus s(V_i)}$ is trivial 
in $H_0(\G)\otimes\Z_2$, 
by what we have shown above, 
$\tau_{U_i\setminus V_i}=\tau_{U_i}\tau_{V_i}$ belongs to 
the commutator subgroup $D([[\G]])$. 
By Theorem \ref{almstfin} (3), 
there exists a compact open $\G$-set $V'$ 
such that $s(V')=s(V_1)$ and $r(V')=s(V_2)$. 
Hence, in the same way as above, 
one can show that $\tau_{V_1}$ is conjugate to $\tau_{V_2}$. 
Therefore, 
we get $[\tau_{U_1}]=[\tau_{V_1}]=[\tau_{V_2}]=[\tau_{U_2}]$ in $[[\G]]_\ab$. 

We define 
the homomorphism $j:H_0(\G)\otimes\Z_2\to[[\G]]_\ab$ as follows. 
Let $f\in C(\G^{(0)},\Z)$. 
Using Lemma \ref{H_0mod2}, 
one can find a non-empty compact open $\G$-set $U$ 
such that $r(U)\cap s(U)=\emptyset$ and 
$[f]=[1_{s(U)}]$ in $H_0(\G)\otimes\Z_2$. 
We set $j([f])=[\tau_U]$. 
By the claim above, the map $j$ is well-defined. 
Moreover, for any $f_1,f_2\in C(\G^{(0)},\Z)$, 
we may choose the $\G$-sets $U_1,U_2$ 
so that $(r(U_1)\cup s(U_1))\cap(r(U_2)\cup s(U_2))=\emptyset$. 
Then 
\begin{align*}
j([f_1]+[f_2])&=j([f_1+f_2])=[\tau_{U_1\cup U_2}]\\
&=[\tau_{U_1}\tau_{U_2}]=[\tau_{U_1}]+[\tau_{U_2}]\\
&=j([f_1])+j([f_2]), 
\end{align*}
which means that $j$ is a homomorphism. 

Let us prove that $\Ker I$ is contained in the image of $j$. 
Assume that $\gamma\in[[\G]]$ is in $\Ker I=[[\G]]_0$. 
By Theorem \ref{almstfin} (4), 
$\gamma$ is a product of transpositions. 
Consequently, 
the equivalence class of $\gamma$ is contained in the image of $j$. 
\end{proof}

In \cite[Lemma 6.3]{M12PLMS}, 
it was shown that 
when $\phi:\Z^N\curvearrowright X$ is a free action of $\Z^N$ 
on a Cantor set $X$, 
the transformation groupoid $\G_\phi$ is almost finite. 
When $N=1$ and $\phi:\Z\curvearrowright X$ is minimal, 
it is well known that 
$H_0(\G_\phi)\cong K_0(C^*_r(\G_\phi))$, 
$H_1(\G_\phi)\cong K_1(C^*_r(\G_\phi))\cong\Z$ 
and $H_n(\G_\phi)=0$ for $n\geq2$. 
Hence, Conjecture \ref{HK} holds for $\G_\phi$. 
Furthermore, the strong AH property also holds 
(\cite[Theorem 4.8]{M06IJM}). 
When $N$ is greater than one, 
it is not known whether Conjecture \ref{HK} is true for $\G_\phi$. 
We remark that the Chern character implies the isomorphisms 
\[
\bigoplus_{i=0}^\infty H_{2i}(\G_\phi)\otimes\Q
\cong K_0(C^*_r(\G_\phi))\otimes\Q
\]
and 
\[
\bigoplus_{i=0}^\infty H_{2i+1}(\G_\phi)\otimes\Q
\cong K_1(C^*_r(\G_\phi))\otimes\Q
\]
(see \cite[Section 4]{FH99ETDS}, \cite[Section 3.1]{M12PLMS}). 

By Theorem \ref{almstfin>AH}, 
the AH conjecture holds for $\G_\phi$. 
But, we do not know 
whether the groupoid $\G_\phi$ has the strong AH property. 

In general, it is not known 
if every almost finite groupoid $\G$ satisfies the strong AH property or not. 
In other words, we cannot prove that 
the homomorphism $j:H_0(\G)\otimes\Z_2\to[[\G]]_\ab$ is always injective, 
and cannot find an example such that the map $j$ has nontrivial kernel. 
(For purely infinite groupoids, we know that $j$ is not always injective. 
See \cite[Example 7.1]{Ne15} and Section 5.5 of the present paper.) 
However, for a minimal free action $\phi:\Z^N\curvearrowright X$, 
one can prove that the kernel of $j$ is at least `contained' 
in the infinitesimal subgroup of $H_0(\G_\phi)$ 
(Proposition \ref{signature}). 

Let us recall the definition of the de la Harpe-Skandalis determinant 
from \cite{HS84AnnInst}. 
Let $A$ be a unital $C^*$-algebra and 
let $T(A)$ be the space of tracial states on $A$. 
The de la Harpe-Skandalis determinant is a homomorphism 
\[
\Delta:U(A)_0\to\Aff(T(A))/D_A(K_0(A)), 
\]
where $U(A)_0$ is the connected component of the identity 
in the unitary group $U(A)$ of $A$ and 
$D_A:K_0(A)\to\Aff(T(A))$ is the homomorphism defined by 
$D_A([p])(\sigma)=\sigma(p)$. 
For $u\in U(A)_0$, $\Delta(u)$ is defined as follows. 
Let $\xi:[0,1]\to U(A)_0$ be a piecewise smooth path 
such that $\xi(0)=1$ and $\xi(1)=u$. 
Then the function 
\[
T(A)\ni\sigma\mapsto
\frac{1}{2\pi\sqrt{-1}}\int_0^1\sigma(\xi'(t)\xi(t)^*)\,dt\in\R
\]
belongs to $\Aff(T(A))$, 
and $\Delta(u)$ is defined as the equivalence class of 
this affine function in $\Aff(T(A))/D_A(K_0(A))$. 
It is known that $\Delta$ is a well-defined homomorphism. 

Let $\G$ be an essentially principal \'etale groupoid 
whose unit space is a Cantor set. 
For any $\mu\in M(\G)$, 
there exists a tracial state $\sigma_\mu$ on $C^*_r(\G)$ 
such that $\sigma_\mu(1_C)=\mu(C)$ for every clopen set $C\subset\G^{(0)}$. 
The map $\mu\mapsto\sigma_\mu$ gives an isomorphism 
between $M(\G)$ and $T(C^*_r(\G))$. 
The infinitesimal subgroup of $H_0(\G)$ is defined by 
\[
\Inf(H_0(\G))
=\left\{[f]\in H_0(\G)\mid\int f\,d\mu=0\quad\forall\mu\in M(\G)\right\}. 
\]

\begin{proposition}\label{signature}
Let $\phi:\Z^N\curvearrowright X$ be a minimal free action of $\Z^N$ 
on a Cantor set $X$. 
The kernel of the homomorphism $j:H_0(\G_\phi)\otimes\Z_2\to[[\G_\phi]]_\ab$ 
is contained in $\Inf(H_0(\G_\phi))\otimes\Z_2$. 
In particular, when $\Inf(H_0(\G_\phi))\otimes\Z_2$ is trivial, 
$\G_\phi$ has the strong AH property. 
\end{proposition}
\begin{proof}
Because $\G_\phi$ is principal, minimal and almost finite, 
Conjecture \ref{AH} holds for $\G_\phi$ by Theorem \ref{almstfin>AH}. 
Put $A=C^*_r(\G_\phi)$. 

Let $U\subset\G_\phi$ be a compact open $\G_\phi$-set 
such that $r(U)\cap s(U)=\emptyset$. 
Set 
\[
V=U\cup U^{-1}\cup(\G_\phi^{(0)}\setminus(r(U)\cup s(U))), 
\]
so that $\theta(V)\in[[\G_\phi]]$ is the transposition 
corresponding to $[1_{s(U)}]$. 
Suppose that $\theta(V)$ is in $D([[\G_\phi]])$. 
It suffices to show that 
the equivalence class of $1_{s(U)}$ is in $\Inf(H_0(\G_\phi))\otimes\Z_2$. 

By \cite[Proposition 5.6]{M12PLMS}, there exists a short exact sequence 
\[
\begin{CD}
1@>>>U(C(X))@>>>N(C(X),A)@>>>[[\G_\phi]]@>>>1, 
\end{CD}
\]
where $N(C(X),A)$ denotes 
the group of unitary normalizers of $C(X)$ in $U(A)$. 
Moreover, 
the homomorphism $N(C(X),A)\to[[\G_\phi]]$ has the right inverse 
$\rho:[[\G_\phi]]\to N(C(X),A)$ defined by $\rho(\theta(W))=1_W$. 
By Theorem \ref{almstfin} (4), $\G_\phi$ has property TR. 
Hence any element of $[[\G_\phi]]_0$ is a product of transpositions, 
and so we get $\rho([[\G_\phi]]_0)\subset U(A)_0$. 
Thus 
\[
\Delta\circ\rho:[[\G_\phi]]_0\to\Aff(T(A))/D_A(K_0(A))
\]
is a well-defined homomorphism. 
Thanks to \cite[Theorem 4.7]{M15crelle}, 
we can conclude that $D([[\G_\phi]])$ is contained 
in the kernel of $\Delta\circ\rho$. 
Therefore we obtain $0=\Delta(\rho(\theta(V)))=\Delta(1_V)$. 
The unitary $1_V\in U(A)_0$ is clearly conjugate to 
the unitary $u=-1_{s(U)}+(1-1_{s(U)})\in U(C(X))$. 
Define $\xi:[0,1]\to U(A)_0$ 
by $\xi(t)=e^{\pi\sqrt{-1}t}1_{s(U)}+(1-1_{s(U)})$ 
so that $\xi(0)=1$ and $\xi(1)=u$. 
One has 
\[
\frac{1}{2\pi\sqrt{-1}}\int_0^1\sigma_\mu(\xi'(t)\xi(t)^*)\,dt
=\frac{1}{2}\mu(s(U)). 
\]
for all $\mu\in M(\G)$. 
It follows that 
the affine function $\sigma_\mu\mapsto\mu(s(U))/2$ belongs to 
the image of the dimension map $D_A$. 
The gap labeling theorem (\cite{BBG06CMP,BO03,KP03Michigan}) tells us that 
$D_A(K_0(A))$ equals $D_A(K_0(C(X)))$. 
As a result, there exists $f\in C(X,\Z)$ such that 
\[
\frac{1}{2}\mu(s(U))=\int_Xf\,d\mu
\]
for all $\mu\in M(\G)$. 
Thus, $[1_{s(U)}]-2[f]$ is in the infinitesimal subgroup $\Inf(H_0(\G_\phi))$, 
and whence $[1_{s(U)}]\otimes\bar{1}$ is in $\Inf(H_0(\G_\phi))\otimes\Z_2$. 
\end{proof}

\begin{remark}
Let $\G$ be a principal, minimal, almost finite groupoid. 
It is natural to ask if the gap labeling theorem holds for $\G$. 
Put $A=C^*_r(\G)$ and consider $D_A:K_0(A)\to\Aff(T(A))$. 
Let $\iota:C(\G^{(0)})\to A$ be the inclusion map. 
Evidently we have $D_A(\iota_*(K_0(C(\G^{(0)})))\subset D_A(K_0(A))$. 
The gap labeling asks whether or not the other inclusion holds. 
If this is the case, the proposition above is true for $\G$. 
\end{remark}

%%%%%%%%%%%%%%%%%%%%%%%%%%%%%%%%%%%%%%%%%%%%%%%%%%%%%%%%%%%%
\section{Purely infinite groupoids}

In this section we discuss purely infinite groupoids. 

\begin{definition}[{\cite[Definition 4.9]{M15crelle}}]
Let $\G$ be an essentially principal \'etale groupoid 
whose unit space is a Cantor set. 
\begin{enumerate}
\item A clopen set $A\subset\G^{(0)}$ is said to be properly infinite 
if there exist compact open $\G$-sets $U,V\subset\G$ 
such that $s(U)=s(V)=A$, $r(U)\cup r(V)\subset A$ 
and $r(U)\cap r(V)=\emptyset$. 
\item We say that $\G$ is purely infinite 
if every clopen set $A\subset\G^{(0)}$ is properly infinite. 
\end{enumerate}
\end{definition}

It is easy to see that if $\G$ is purely infinite, then 
for any clopen set $A\subset\G^{(0)}$ 
the reduction $\G|A$ is also purely infinite. 
It is also clear that if $\G^{(0)}$ is properly infinite, 
then there does not exist a $\G$-invariant probability measure. 
When $\G^{(0)}$ is properly infinite, 
the topological full group $[[\G]]$ contains 
a subgroup isomorphic to the free product $\Z_2*\Z_3$ 
(\cite[Proposition 4.10]{M15crelle}). 
In particular, $[[\G]]$ is not amenable. 

\begin{theorem}\label{pi}
Let $\G$ be a purely infinite groupoid. 
\begin{enumerate}
\item If $\G$ is minimal, then $D([[\G]])$ is simple. 
\item The index map $I:[[\G]]\to H_1(\G)$ is surjective. 
\item If $\G$ is minimal, then $\G$ has cancellation. 
\end{enumerate}
\end{theorem}
\begin{proof}
(1) is \cite[Theorem 4.16]{M15crelle}. 
(2) is \cite[Theorem 5.2]{M15crelle}. 

Let us prove (3). 
Suppose that 
two clopen sets $A,B\subset\G^{(0)}$ satisfy $[1_A]=[1_B]$ in $H_0(\G)$. 
Since $\G$ is minimal, 
there exists a compact open $\G$-set $W\subset\G$ 
such that $r(W)\subset B$ and $s(W)\subset A$. 
Let $A'=A\setminus s(W)$ and $B'=B\setminus r(W)$. 
Then, $[1_{A'}]=[1_A]-[1_{s(W)}]=[1_B]-[1_{r(W)}]=[1_{B'}]$. 
Hence there exist compact open $\G$-sets $U_1,U_2,\dots,U_n$ 
such that 
\[
1_{A'}-1_{B'}=\sum_{i=1}^n\left(1_{r(U_i)}-1_{s(U_i)}\right). 
\]
Put 
\[
f=1_{A'}+\sum_{i=1}^n1_{s(U_i)}=1_{B'}+\sum_{i=1}^n1_{r(U_i)}. 
\]
Let $m=\max\{f(x)\mid x\in\G^{(0)}\}$. 
We can choose compact open $\G$-sets $V_1,V_2,\dots,V_m$ so that 
\[
s(V_j)=\{x\in\G^{(0)}\mid j\leq f(x)\leq m\}
\]
and $r(V_j)$ are mutually disjoint, 
because $\G$ is purely infinite and minimal. 
Let $E=\bigcup_lr(V_l)$. 

For $i=1,2,\dots,n$ and $j=1,2,\dots,m$, 
we define clopen sets $C_{i,j}$ and $D_{i,j}$ by 
\[
C_{i,j}=\left\{x\in s(U_i)\mid
\left(1_{A'}+\sum_{k=1}^i1_{s(U_k)}\right)(x)=j\right\}
\]
and 
\[
D_{i,j}=\left\{x\in r(U_i)\mid
\left(1_{B'}+\sum_{k=1}^i1_{r(U_k)}\right)(x)=j\right\}. 
\]
For any $i$, it is easy to see that 
$\{C_{i,j}\mid j=1,2,\dots,m\}$ is a clopen partition of $s(U_i)$ and 
that $\{D_{i,j}\mid j=1,2,\dots,m\}$ is a clopen partition of $r(U_i)$. 
For any $j$, we can verify 
$C_{i,j}\cap C_{k,j}=D_{i,j}\cap D_{k,j}=\emptyset$ whenever $i\neq k$. 
Moreover, 
\[
\bigcup_{i=1}^nC_{i,j}
=\begin{cases}s(V_1)\setminus A'&j=1\\
s(V_j)&j>1\end{cases}
\]
and 
\[
\bigcup_{i=1}^nD_{i,j}
=\begin{cases}r(V_1)\setminus B'&j=1\\
r(V_j)&j>1. \end{cases}
\]
Consider 
\[
\widetilde{V}=\bigcup_{i,k,l}V_lD_{i,l}U_iC_{i,k}V_k^{-1}. 
\]
Evidently $\widetilde{V}$ is a $\G$-set, and 
\begin{align*}
\widetilde{V}\widetilde{V}^{-1}
&=\left(\bigcup_{i,k,l}V_lD_{i,l}U_iC_{i,k}V_k^{-1}\right)
\left(\bigcup_{j,p,q}V_pC_{j,p}U_j^{-1}D_{j,q}V_q^{-1}\right)\\
&=\bigcup_{i,j,k,l,q}V_lD_{i,l}U_iC_{i,k}s(V_k)C_{j,k}U_j^{-1}D_{j,q}V_q^{-1}\\
&=\bigcup_{i,k,l,q}V_lD_{i,l}U_iC_{i,k}U_i^{-1}D_{i,q}V_q^{-1}\\
&=\bigcup_{i,l,q}V_lD_{i,l}U_is(U_i)U_i^{-1}D_{i,q}V_q^{-1}\\
&=\bigcup_{i,l,q}V_lD_{i,l}r(U_i)D_{i,q}V_q^{-1}\\
&=\bigcup_{i,l}V_lD_{i,l}V_l^{-1}\\
&=\left(\bigcup_lr(V_l)\right)\setminus V_1B'V_1^{-1}
=E\setminus V_1B'V_1^{-1}. 
\end{align*}
Similarly, we obtain 
$\widetilde{V}^{-1}\widetilde{V}=E\setminus V_1A'V_1^{-1}$. 
Since $\G$ is purely infinite and minimal, 
there exists a compact open $\G$-set $T$ 
such that $s(T)\subset s(W)$ and $r(T)=s(\tilde V)$. 
Define a compact open $\G$-set $U$ by 
\[
U=(B'V_1^{-1}\cup WT^{-1}\widetilde V^{-1})(V_1A'\cup T). 
\]
Then we get 
\begin{align*}
U^{-1}U
&=(A'V_1^{-1}\cup T^{-1})
(V_1B'V_1^{-1}\cup r(\widetilde V))(V_1A'\cup T)\\
&=(A'V_1^{-1}\cup T^{-1})E(V_1A'\cup T)\\
&=A'\cup s(T)
\end{align*}
and 
\begin{align*}
UU^{-1}
&=(B'V_1^{-1}\cup WT^{-1}\widetilde V^{-1})
(V_1A'V_1^{-1}\cup s(\widetilde V))(V_1B'\cup \widetilde VTW^{-1})\\
&=(B'V_1^{-1}\cup WT^{-1}\widetilde V^{-1})E
(V_1B'\cup \widetilde VTW^{-1})\\
&=B'\cup Ws(T)W^{-1}. 
\end{align*}
Therefore, 
$\widetilde U=U\cup W(s(W)\setminus s(T))$ is a compact open $\G$-set 
satisfying $s(\widetilde U)=A$ and $r(\widetilde U)=B$. 
\end{proof}

\begin{lemma}\label{H_0ofpi}
Let $\G$ be a minimal purely infinite groupoid. 
Let $f\in C(\G^{(0)},\Z)$. 
For any non-empty clopen set $A\subset\G^{(0)}$, 
there exists a clopen subset $B\subset A$ 
such that $f$ and $1_B$ have the same equivalence class in $H_0(\G)$. 
\end{lemma}
\begin{proof}
This immediately follows 
from \cite[Lemma 5.3, Proposition 4.11]{M15crelle}. 
\end{proof}

In the same way as Theorem \ref{almstfin>AH} 
we can prove the following. 

\begin{theorem}\label{pi>AH}
Let $\G$ be a minimal purely infinite groupoid. 
If $\G$ has property TR, then 
Conjecture \ref{AH} holds for $\G$. 
\end{theorem}
\begin{proof}
The proof of Theorem \ref{almstfin>AH} works 
by using Theorem \ref{pi} and Lemma \ref{H_0ofpi} 
instead of Theorem \ref{almstfin} and Lemma \ref{H_0mod2}. 
\end{proof}

It is not known if all minimal purely infinite groupoids have TR. 
Here, we describe a situation 
where property TR is inherited from a smaller groupoid to a larger groupoid. 

Our setting is as follows. 
Let $\G$ be a minimal \'etale groupoid. 
Let $c:\G\to\Z$ be a continuous surjective homomorphism and 
let $\H=\Ker c$, 
which is a subgroupoid of $\G$ with $\H^{(0)}=\G^{(0)}$. 
Choose a non-empty compact open $\G$-set $U\subset c^{-1}(1)$. 
Then, $U$ induces an isomorphism $\pi:\H|r(U)\to\H|s(U)$ 
such that $\{\pi(h)\}=U^{-1}hU$. 
By \cite[Proposition 3.5, Theorem 3.6]{M12PLMS}, 
there exist natural isomorphisms 
$H_n(\H)\cong H_n(\H|r(U))\cong H_n(\H|s(U))$. 
Therefore, we obtain isomorphisms 
\[
\delta_n:H_n(\H)\cong H_n(\H|r(U))
\stackrel{H_n(\pi)}{\longrightarrow}H_n(\H|s(U))\cong H_n(\H). 
\]
By \cite[Proposition 3.7, Theorem 3.8]{M12PLMS}, 
there exists a spectral sequence 
\[
E^2_{p,q}=H_p(\Z,H_q(\H))\Rightarrow H_{p+q}(\G). 
\]
Hence, $H_0(\G)\cong\Coker(\id-\delta_0)$, 
and for any $n\geq1$ 
there exists a short exact sequence 
\begin{equation}
\begin{CD}
0@>>>\Coker(\id-\delta_n)@>>>H_n(\G)@>>>
\Ker(\id-\delta_{n-1})@>>>0. 
\end{CD}
\label{cow}
\end{equation}

For $f\in C(G^{(0)},\Z)$, 
we denote its equivalence class in $H_0(\H)$ by $[f]_\H$. 
When $V\subset c^{-1}(n)$ is a compact open $\G$-set, 
we have $\delta_0^n([1_{r(V)}]_\H)=[1_{s(V)}]_\H$. 

\begin{proposition}\label{GhasTR}
Let $\G$ be a minimal \'etale groupoid. 
Let $c:\G\to\Z$ be a continuous surjective homomorphism and 
let $\H=\Ker c$. 
Assume either of the following conditions. 
\begin{enumerate}
\item $\H$ is a principal, minimal, almost finite groupoid 
with $M(\H)=\{\mu\}$, 
and there exists a real number $0<\lambda<1$ such that, 
for any compact open $\G$-set $U\subset c^{-1}(1)$, 
$\mu(r(U))=\lambda\mu(s(U))$ holds. 
\item $\H$ is a minimal, purely infinite groupoid 
satisfying property TR. 
\end{enumerate}
Then, $\G$ is purely infinite and satisfies property TR. 
\end{proposition}

In order to prove this proposition, we need two lemmas. 
Write $X=\G^{(0)}$. 
For $\alpha\in\Homeo(X)$, we let $\supp(\alpha)$ be 
the closure of $\{x\in X\mid\alpha(x)\neq x\}$. 
Let $\mathcal{T}\subset[[\G]]_0$ denote the set of all transpositions and 
let $\langle\mathcal{T}\rangle\subset[[\G]]_0$ denote 
the subgroup generated by $\mathcal{T}$. 

\begin{lemma}\label{LemAforGhasTR}
Under the same hypothesis as Proposition \ref{GhasTR}, the following holds. 
For any $n\in\Z$ and any non-empty clopen sets $A,B\subset X$, 
if $\delta_0^n([1_A]_\H)=[1_B]_\H$, 
then there exists a compact open $\G$-set $V\subset c^{-1}(n)$ 
such that $A=r(V)$, $B=s(V)$. 
\end{lemma}
\begin{proof}
Let us consider the case (1). 
Choose a non-empty compact open $\G$-set $U\subset c^{-1}(n)$. 
Since $\H$ is minimal, 
there exist compact open $\H$-sets $W_1,W_2,\dots,W_k\subset\H$ 
such that $\{r(W_1),r(W_2),\dots,r(W_k)\}$ is a partition of $A$ 
and $s(W_i)\subset r(U)$. 
We have 
\[
\sum_{i=1}^k[1_{s(W_iU)}]_\H=\sum_{i=1}^k\delta_0^n([1_{r(W_iU)}]_\H)
=\sum_{i=1}^k\delta_0^n([1_{r(W_i)}]_\H)=\delta_0^n([1_A]_\H)=[1_B]_\H. 
\]
It follows from \cite[Lemma 6.7]{M12PLMS} and Theorem \ref{almstfin} (3) that 
we can find compact open $\H$-sets $W'_1,W'_2,\dots,W'_k\subset\H$ 
such that $\{s(W'_1),s(W'_2),\dots,s(W'_k)\}$ is a partition of $B$ 
and $s(W_iU)=r(W'_i)$. 
Then $V=\bigcup_iW_iUW'_i$ is a compact open $\G$-set 
such that $V\subset c^{-1}(n)$, $r(V)=A$ and $s(V)=B$. 

In the case (2), 
almost the same proof works 
by replacing \cite[Lemma 6.7]{M12PLMS} and Theorem \ref{almstfin} (3) 
with \cite[Proposition 4.11]{M15crelle} and Theorem \ref{pi} (3). 
\end{proof}

\begin{lemma}\label{LemBforGhasTR}
Under the same hypothesis as Proposition \ref{GhasTR}, the following holds. 
If $\alpha\in[[\H]]$ satisfies $I(\alpha)=0$ in $H_1(\G)$, 
then $\alpha$ is in $\langle\mathcal{T}\rangle$. 
\end{lemma}
\begin{proof}
Let us consider the case (1). 
Choose a non-empty compact open $\G$-set $U\subset c^{-1}(1)$. 
By replacing $U$ with its suitable clopen subset, 
we may assume that $r(U)\cap s(U)=\emptyset$. 
We let $I_\H:[[\H]]\to H_1(\H)$ denote 
the index map for $\H$. 
Since $I(\alpha)=0$ in $H_1(\G)$, 
from the short exact sequence \eqref{cow}, 
we conclude that $I_\H(\alpha)$ is trivial in $\Coker(\id-\delta_1)$. 
Thus, 
there exists $z\in H_1(\H)$ such that $I_\H(\alpha)=z-\delta_1(z)$. 
It follows from Theorem \ref{almstfin} (2) that 
there exists $\beta\in[[\H]]$ such that $I_\H(\beta)=z$. 
By \cite[Lemma 7.10]{M12PLMS} and its proof, 
we can find an elementary homeomorphism $\beta_0\in[[H]]$ 
(see \cite[Definition 7.6 (1)]{M12PLMS}) 
such that $\mu(\supp(\beta_0\beta))<\mu(r(U))$. 
It follows from \cite[Lemma 6.7]{M12PLMS} that, 
by replacing $U$ again, we may assume $\supp(\beta_0\beta)\subset r(U)$. 
Define the transposition $\tau_U\in\mathcal{T}$ 
as in the proof of Theorem \ref{almstfin>AH}. 
Then, by the definition of $\delta_1$, 
one has $I_\H(\tau_U\beta_0\beta\tau_U)=\delta_1(I_\H(\beta_0\beta))$. 
Hence, 
\begin{align*}
I_\H\left((\beta_0\beta)\tau_U(\beta_0\beta)^{-1}\tau_U\right)
&=I_\H(\beta_0\beta)-\delta_1(I_\H(\beta_0\beta))\\
&=I_\H(\beta)-\delta_1(I_\H(\beta))=z-\delta_1(z). 
\end{align*}
Clearly $\gamma=(\beta_0\beta)\tau_U(\beta_0\beta)^{-1}\tau_U$ 
belongs to $\langle\mathcal{T}\rangle$, 
and $I_\H(\alpha\gamma^{-1})=I_\H(\alpha)-(z-\delta_1(z))=0$ 
in $H_1(\H)$. 
Since $\H$ satisfies property TR, 
$\alpha\gamma^{-1}$ is in $\langle\mathcal{T}\rangle$. 
Therefore $\alpha$ is in $\langle\mathcal{T}\rangle$. 

In the case (2), 
the same proof works 
by replacing Theorem \ref{almstfin} (2) and \cite[Lemma 6.7]{M12PLMS} 
with Theorem \ref{pi} (2) and \cite[Proposition 4.11]{M15crelle}. 
\end{proof}

\begin{proof}[Proof of Proposition \ref{GhasTR}]
First, let us consider the case (1). 

We will show that $\G$ is purely infinite. 
Let $A,B\subset X$ be non-empty clopen subsets. 
For sufficiently large $n\in\N$, we have $\lambda^n\mu(A)<\mu(B)$. 
By \cite[Lemma 6.7]{M12PLMS}, 
there exists $C\subset B$ such that $[1_C]_\H=\delta_0^{-n}([1_A]_\H)$. 
By Lemma \ref{LemAforGhasTR}, 
there exists a compact open $\G$-set $U\subset\G$ 
such that $s(U)=A$, $r(U)=C\subset B$. 
By \cite[Proposition 4.11]{M15crelle}, $\G$ is purely infinite. 

We will show that $\G$ has property TR. 
The proof goes exactly as that of \cite[Lemma 6.10]{M15crelle}. 
For $\alpha\in[[\G]]$, 
we take the compact open $\G$-set $U\subset\G$ 
satisfying $\alpha=\theta(U)=(r|U)\circ(s|U)^{-1}$, 
and define clopen subsets $S(\alpha,n)\subset X$ by 
$S(\alpha,n)=s(U\cap c^{-1}(n))$. 
Note that $S(\alpha,n)$ is empty except for finitely many $n\in\Z$. 

Suppose that $\alpha\in[[\G]]_0\setminus\{1\}$ is given. 
In the same way as \cite[Lemma 6.10]{M15crelle}, 
we may assume that $A=\supp(\alpha)$ is not equal to $X$. 
We have 
\[
[1_X]_\H=\sum_{n\in\Z}[1_{S(\alpha,n)}]_\H
\]
in $H_0(\H)$, 
because $\{S(\alpha,n)\mid n\in\Z\}$ is a clopen partition of $X$. 
Since $\{\alpha(S(\alpha,n))\mid n\in\Z\}$ is also a clopen partition of $X$, 
one obtains 
\[
[1_X]_\H=\sum_{n\in\Z}[1_{\alpha(S(\alpha,n))}]_\H
=\sum_{n\in\Z}\delta_0^{-n}\left([1_{S(\alpha,n)}]_\H\right)
\]
in $H_0(\H)$. 
Therefore, we get 
\begin{equation}
\sum_{n\in\Z}(\id-\delta_0^{-n})\left([1_{S(\alpha,n)}]_\H\right)=0. 
\label{mouse}
\end{equation}
For $n\in\Z$, we define homomorphisms $\delta_0^{(n)}:H_0(\H)\to H_0(\H)$ by 
\[
\delta_0^{(n)}=\begin{cases}\id+\delta_0+\dots+\delta_0^{n-1}&n>0\\
0&n=0\\
-(\delta_0^{-1}+\delta_0^{-2}+\dots+\delta_0^n)&n<0, \end{cases}
\]
so that $(\id-\delta_0)\delta_0^{(n)}=\id-\delta_0^n$ hold. 
It follows from \eqref{mouse} that 
\[
\sum_{n\in\Z}(\id-\delta_0)
\left(\delta_0^{(-n)}\left([1_{S(\alpha,n)}]_\H\right)\right)=0, 
\]
which means that $\sum_{n\in\Z}\delta_0^{(-n)}([1_{S(\alpha,n)}]_\H)$ 
is in $\Ker(\id-\delta_0)$. 
The proof of \cite[Theorem 4.14]{M12PLMS} implies that, 
in the exact sequence \eqref{cow}, 
the element $I(\alpha)\in H_1(\G)$ maps to 
$\sum_n\delta_0^{(-n)}([1_{S(\alpha,n)}]_\H)\in\Ker(\id-\delta_0)$ 
(see \cite[Lemma 6.8]{M15crelle}). 
Since we assumed $I(\alpha)=0$, we get 
\[
\sum_{n\in\Z}\delta_0^{(-n)}([1_{S(\alpha,n)}]_\H)=0
\]
in $H_0(\H)$. 
Set $P=\{n\in\N\mid S(\alpha,n)\neq\emptyset\}$ and 
$Q=\{n\in\Z\mid n<0,\ S(\alpha,n)\neq\emptyset\}$. 
Then 
\begin{equation}
-\sum_{n\in P}\delta_0^{(-n)}\left([1_{S(\alpha,n)}]_\H\right)=
\sum_{n\in Q}\delta_0^{(-n)}\left([1_{S(\alpha,n)}]_\H\right). 
\label{tiger}
\end{equation}
Put 
\[
z=[1_A]_\H+\delta_0\left(
\sum_{n\in Q}\delta_0^{(-n)}\left([1_{S(\alpha,n)}]_\H\right)\right)
\in H_0(\H). 
\]
We have 
\[
\hat\mu(z)=\mu(A)
+\sum_{n\in Q}(\lambda^{-1}+\lambda^{-2}+\dots+\lambda^n)\mu(S(\alpha,n)), 
\]
where the homomorphism $\hat\mu:H_0(\H)\to\R$ is defined by 
\[
\hat\mu([f]_\H)=\int f\,d\mu
\]
for $f\in C(X,\Z)$. 
Choose $m\in\N$ so that 
\[
\lambda^m\mu(A)
<\mu(X\setminus A)\quad\text{and}\quad \lambda^m\hat\mu(z)<1
\]
hold. 
Then we can find a clopen set $B\subset X\setminus A$ 
such that $[1_B]_\H=\delta_0^{-m}([1_A]_\H)$. 
It follows from Lemma \ref{LemAforGhasTR} that 
there exists a compact open $\G$-set $V\subset c^{-1}(m)$ 
such that $r(V)=B$ and $s(V)=A$. 
Define the transposition $\tau_V\in\mathcal{T}$ 
as in the proof of Theorem \ref{almstfin>AH}. 
Set $\beta=\tau_V\alpha\tau_V$. 
It suffices to show that $\beta$ belongs to $\langle\mathcal{T}\rangle$. 
Notice that $\supp(\beta)=\tau_V(A)=B$, $S(\beta,n)=\tau_V(S(\alpha,n))$ 
and 
\[
[1_{S(\beta,n)}]_\H=\delta^{-m}\left([1_{S(\alpha,n)}]_\H\right)
\]
for every $n\in\Z$. 
Hence, by \eqref{tiger}, we get 
\begin{equation}
-\sum_{n\in P}\delta_0^{(-n)}\left([1_{S(\beta,n)}]_\H\right)=
\sum_{n\in Q}\delta_0^{(-n)}\left([1_{S(\beta,n)}]_\H\right). 
\end{equation}

Then, in the same way as the proof of \cite[Lemma 6.10]{M15crelle}, 
we can construct $\gamma\in\langle\mathcal{T}\rangle$ 
such that $S(\gamma,n)=S(\beta,n)$ for all $n\in\Z$. 
Then we obtain $S(\beta\gamma^{-1},0)=X$, 
which implies $\beta\gamma^{-1}\in[[\H]]$. 
Since $I(\beta\gamma^{-1})=I(\beta)=I(\alpha)=0$, 
by Lemma \ref{LemBforGhasTR}, 
we get $\beta\gamma^{-1}\in\langle\mathcal{T}\rangle$. 
It follows that $\beta$ belongs to $\langle\mathcal{T}\rangle$, 
as desired. 

In the case (2), 
almost the same proof works 
by replacing \cite[Lemma 6.7]{M12PLMS} and Theorem \ref{almstfin} (3) 
with \cite[Proposition 4.11]{M15crelle} and Theorem \ref{pi} (3). 
We omit it. 
\end{proof}

%%%%%%%%%%%%%%%%%%%%%%%%%%%%%%%%%%%%%%%%%%%%%%%%%%%%%%%%%%%%
\section{Products of SFT groupoids}

%%%%%%%%%%%%%%%%%%%%%%%%%%%%%%%%%%%%%%%%%%%%%%%%%%%%%%%%%%%%
\subsection{Preliminaries}

We first recall the definition of 
\'etale groupoids arising from one-sided shifts of finite type. 
Let $(\mathcal{V},\mathcal{E})$ be a finite directed graph, 
where $\mathcal{V}$ is a finite set of vertices 
and $\mathcal{E}$ is a finite set of edges. 
For $e\in\mathcal{E}$, $i(e)$ denotes the initial vertex of $e$ and 
$t(e)$ denotes the terminal vertex of $e$. 
Let $A=(A(\xi,\eta))_{\xi,\eta\in\mathcal{V}}$ be 
the adjacency matrix of $(\mathcal{V},\mathcal{E})$, 
that is, 
\[
A(\xi,\eta)=\#\{e\in\mathcal{E}\mid i(e)=\xi,\ t(e)=\eta\}. 
\]
We assume that $A$ is irreducible 
(i.e. for all $\xi,\eta\in\mathcal{V}$ 
there exists $n\in\N$ such that $A^n(\xi,\eta)>0$) 
and that $A$ is not a permutation matrix. 
Define 
\[
X_A=\{(x_k)_{k\in\N}\in\mathcal{E}^\N
\mid t(x_k)=i(x_{k+1})\quad\forall k\in\N\}. 
\]
With the product topology, $X_A$ is a Cantor set. 
Define a surjective continuous map $\sigma_A:X_A\to X_A$ by 
\[
\sigma_A(x)_k=x_{k+1}\quad k\in\N,\ x=(x_k)_k\in X_A. 
\]
In other words, $\sigma_A$ is the (one-sided) shift on $X_A$.  
It is easy to see that $\sigma_A$ is a local homeomorphism. 
The dynamical system $(X_A,\sigma_A)$ is called 
the one-sided irreducible shift of finite type (SFT) 
associated with the graph $(\mathcal{V},\mathcal{E})$ (or the matrix $A$). 

The \'etale groupoid $\G_A$ for $(X_A,\sigma_A)$ is given by 
\[
\G_A=\{(x,n,y)\in X_A\times\Z\times X_A\mid
\exists k,l\in\N,\ n=k{-}l,\ \sigma_A^k(x)=\sigma_A^l(y)\}. 
\]
The topology of $\G_A$ is generated by the sets 
$\{(x,k{-}l,y)\in\G_A\mid x\in P,\ y\in Q,\ \sigma_A^k(x)=\sigma_A^l(y)\}$, 
where $P,Q\subset X_A$ are open and $k,l\in\N$. 
Two elements $(x,n,y)$ and $(x',n',y')$ in $\G$ are composable 
if and only if $y=x'$, and the multiplication and the inverse are 
\[
(x,n,y)\cdot(y,n',y')=(x,n{+}n',y'),\quad (x,n,y)^{-1}=(y,-n,x). 
\]
We identify $X_A$ with the unit space $\G_A^{(0)}$ via $x\mapsto(x,0,x)$. 
We call $\G_A$ the SFT groupoid associated with the matrix $A$. 

The following is a classification theorem of SFT groupoids. 
For an $N\times N$ matrix $A$ with entries in $\N\cup\{0\}$, 
the Bowen-Franks group $\BF(A)$ is the abelian group $\Z^N/(\id-A)\Z^N$. 
We let $u_A$ denote 
the equivalence class of $(1,1,\dots,1)\in\Z^N$ in $\BF(A^t)$. 

\begin{theorem}[{\cite[Theorem 3.6]{MM14Kyoto}}]\label{SFTclassify}
Let $(X_A,\sigma_A)$ and $(X_B,\sigma_B)$ be 
two irreducible one-sided shifts of finite type. 
The following conditions are equivalent. 
\begin{enumerate}
\item $(X_A,\sigma_A)$ and $(X_B,\sigma_B)$ are 
continuously orbit equivalent. 
\item The \'etale groupoids $\G_A$ and $\G_B$ are isomorphic. 
\item There exists an isomorphism $\Phi:\BF(A^t)\to\BF(B^t)$ 
such that $\Phi(u_A)=u_B$ and $\sgn(\det(\id-A))=\sgn(\det(\id-B))$. 
\end{enumerate}
\end{theorem}

The following is an immediate consequence of the theorem above. 

\begin{corollary}\label{SFTMorita}
Let $(X_A,\sigma_A)$ and $(X_B,\sigma_B)$ be 
two irreducible one-sided shifts of finite type. 
The following conditions are equivalent. 
\begin{enumerate}
\item The \'etale groupoids $\G_A$ and $\G_B$ are Morita equivalent. 
\item $\BF(A^t)\cong\BF(B^t)$ and $\sgn(\det(\id-A))=\sgn(\det(\id-B))$. 
\end{enumerate}
\end{corollary}

The following are known for the SFT groupoids. 

\begin{theorem}[{\cite{M15crelle}}]\label{SFTproperty}
Let $(X_A,\sigma_A)$ be an irreducible one-sided shift of finite type 
and let $\G_A$ be the associated SFT groupoid. 
\begin{enumerate}
\item $\G_A$ is purely infinite and minimal. 
\item The homology groups of $\G_A$ are 
\[
H_n(\G_A)\cong\begin{cases}\BF(A^t)&n=0\\
\Ker(\id-A^t)&n=1\\
0&n\geq2, \end{cases}
\]
and the equivalence class of $1_{X_A}$ in $H_0(\G_A)$ 
is equal to $u_A\in\BF(A^t)$. 
In particular, the HK conjecture (Conjecture \ref{HK}) holds for $\G_A$. 
\item $\G_A$ has the strong AH property, and 
the group $[[\G_A]]_\ab$ is isomorphic to 
$(H_0(\G_A)\otimes\Z_2)\oplus H_1(\G_A)$. 
\item $[[\G_A]]$ has the Haagerup property. 
\item $[[\G_A]]$ is of type F$_\infty$. 
\end{enumerate}
\end{theorem}

%%%%%%%%%%%%%%%%%%%%%%%%%%%%%%%%%%%%%%%%%%%%%%%%%%%%%%%%%%%%
\subsection{HK and AH}

For $i=1,2,\dots,n$, let $(X_{A_i},\sigma_{A_i})$ be 
an irreducible one-sided shift of finite type 
and let $\G_{A_i}$ be the associated SFT groupoid. 
We consider the product groupoid 
$\G=\G_{A_1}\times\G_{A_2}\times\dots\times\G_{A_n}$. 
It is clear that $\G$ is purely infinite and minimal. 

The homology groups of $\G$ can be computed 
by using Theorem \ref{SFTproperty} (2) and 
the K\"unneth theorem (Theorem \ref{Kunneth}). 

\begin{proposition}\label{prodSFTH_k}
Let $\G=\G_{A_1}\times\G_{A_2}\times\dots\times\G_{A_n}$ be 
a product groupoid of SFT groupoids. 
Then, 
\begin{align*}
H_k(\G)\cong&\left(\Z^{\binom{n-1}{k}}
\otimes H_0(\G_{A_1})\otimes H_0(\G_{A_2})\otimes
\dots\otimes H_0(\G_{A_n})\right)\\
&\oplus\left(\Z^{\binom{n-1}{k-1}}
\otimes H_1(\G_{A_1})\otimes H_1(\G_{A_2})\otimes
\dots\otimes H_1(\G_{A_n})\right), 
\end{align*}
where $\binom{n}{k}$ denote the binomial coefficients and 
they are understood as zero unless $0\leq k\leq n$. 
The equivalence class of the constant function $1_{\G^{(0)}}$ 
in $H_0(\G)=H_0(\G_{A_1})\otimes\dots\otimes H_0(\G_{A_n})$ 
is $u_{A_1}\otimes\dots\otimes u_{A_n}$. 
\end{proposition}
\begin{proof}
For an abelian group $P$, 
we let $P_{\mathrm{tor}}$ denote the torsion subgroup of $P$. 
When $P$ and $Q$ are finitely generated abelian groups, 
we have $\Tor(P,Q)\cong P_{\mathrm{tor}}\otimes Q_{\mathrm{tor}}$. 
By Theorem \ref{SFTproperty} (2), 
$H_0(\G_{A_i})$ is a finitely generated abelian group and 
$H_1(\G_{A_i})$ is isomorphic to the torsion free part of $H_0(\G_{A_i})$. 
Thus, $H_0(\G_{A_i})\cong H_1(\G_{A_i})\oplus H_0(\G_{A_i})_{\mathrm{tor}}$. 

The proof is by induction on $n$. 
Assume that the proposition is true for $n{-}1$. 
Theorem \ref{Kunneth} implies 
\begin{align*}
H_k(\G_{A_1}\times\dots\times\G_{A_n})
&\cong\left(H_{k-1}(\G_{A_1}\times\dots\times\G_{A_{n-1}})
\otimes H_1(\G_{A_n})\right)\\
&\qquad\oplus\left(H_k(\G_{A_1}\times\dots\times\G_{A_{n-1}})
\otimes H_0(\G_{A_n})\right)\\
&\qquad\oplus\Tor(H_{k-1}(\G_{A_1}\times\dots\times\G_{A_{n-1}}),
H_0(\G_{A_n}))\\
&\cong\left(H_{k-1}(\G_{A_1}\times\dots\times\G_{A_{n-1}})
\otimes H_1(\G_{A_n})\right)\\
&\qquad\oplus\left(H_k(\G_{A_1}\times\dots\times\G_{A_{n-1}})
\otimes H_0(\G_{A_n})\right)\\
&\qquad\oplus\left(
H_{k-1}(\G_{A_1}\times\dots\times\G_{A_{n-1}})_{\mathrm{tor}}
\otimes H_0(\G_{A_n})_{\mathrm{tor}}\right), 
\end{align*}
which is isomorphic to 
\begin{align*}
&\left(\Z^{\binom{n-2}{k-1}}\otimes
H_0(\G_{A_1})\otimes\dots\otimes H_0(\G_{A_{n-1}})
\otimes H_1(\G_{A_n})\right)\\
&\qquad\oplus\left(\Z^{\binom{n-2}{k-2}}\otimes
H_1(\G_{A_1})\otimes\dots\otimes H_1(\G_{A_{n-1}})
\otimes H_1(\G_{A_n})\right)\\
&\qquad\oplus\left(\Z^{\binom{n-2}{k}}\otimes
H_0(\G_{A_1})\otimes\dots\otimes H_0(\G_{A_{n-1}})
\otimes H_0(\G_{A_n})\right)\\
&\qquad\oplus\left(\Z^{\binom{n-2}{k-1}}\otimes
H_1(\G_{A_1})\otimes\dots\otimes H_1(\G_{A_{n-1}})
\otimes H_0(\G_{A_n})\right)\\
&\qquad\oplus\left(\Z^{\binom{n-2}{k-1}}\otimes
(H_0(\G_{A_1})\otimes\dots\otimes H_0(\G_{A_{n-1}}))_{\mathrm{tor}}
\otimes H_0(\G_{A_n})_{\mathrm{tor}}\right)\\
&\cong\left(\Z^{\binom{n-2}{k-1}}\otimes
H_0(\G_{A_1})\otimes\dots\otimes H_0(\G_{A_{n-1}})
\otimes H_1(\G_{A_n})\right)\\
&\qquad\oplus\left(\Z^{\binom{n-2}{k-2}}\otimes
H_1(\G_{A_1})\otimes\dots\otimes H_1(\G_{A_{n-1}})
\otimes H_1(\G_{A_n})\right)\\
&\qquad\oplus\left(\Z^{\binom{n-2}{k}}\otimes
H_0(\G_{A_1})\otimes\dots\otimes H_0(\G_{A_{n-1}})
\otimes H_0(\G_{A_n})\right)\\
&\qquad\oplus\left(\Z^{\binom{n-2}{k-1}}\otimes
H_1(\G_{A_1})\otimes\dots\otimes H_1(\G_{A_{n-1}})
\otimes H_1(\G_{A_n})\right)\\
&\qquad\oplus\left(\Z^{\binom{n-2}{k-1}}\otimes
H_1(\G_{A_1})\otimes\dots\otimes H_1(\G_{A_{n-1}})
\otimes H_0(\G_{A_n})_{\mathrm{tor}}\right)\\
&\qquad\oplus\left(\Z^{\binom{n-2}{k-1}}\otimes
(H_0(\G_{A_1})\otimes\dots\otimes H_0(\G_{A_{n-1}}))_{\mathrm{tor}}
\otimes H_0(\G_{A_n})_{\mathrm{tor}}\right)\\
&\cong\left(\Z^{\binom{n-2}{k-1}}\otimes
H_0(\G_{A_1})\otimes\dots\otimes H_0(\G_{A_{n-1}})
\otimes H_1(\G_{A_n})\right)\\
&\qquad\oplus\left(\Z^{\binom{n-2}{k-2}}\otimes
H_1(\G_{A_1})\otimes\dots\otimes H_1(\G_{A_{n-1}})
\otimes H_1(\G_{A_n})\right)\\
&\qquad\oplus\left(\Z^{\binom{n-2}{k}}\otimes
H_0(\G_{A_1})\otimes\dots\otimes H_0(\G_{A_{n-1}})
\otimes H_0(\G_{A_n})\right)\\
&\qquad\oplus\left(\Z^{\binom{n-2}{k-1}}\otimes
H_1(\G_{A_1})\otimes\dots\otimes H_1(\G_{A_{n-1}})
\otimes H_1(\G_{A_n})\right)\\
&\qquad\oplus\left(\Z^{\binom{n-2}{k-1}}\otimes
H_0(\G_{A_1})\otimes\dots\otimes H_0(\G_{A_{n-1}})
\otimes H_0(\G_{A_n})_{\mathrm{tor}}\right)\\
&\cong\left(\Z^{\binom{n-2}{k-1}}\otimes
H_0(\G_{A_1})\otimes\dots\otimes H_0(\G_{A_{n-1}})
\otimes H_0(\G_{A_n})\right)\\
&\qquad\oplus\left(\Z^{\binom{n-2}{k-2}}\otimes
H_1(\G_{A_1})\otimes\dots\otimes H_1(\G_{A_{n-1}})
\otimes H_1(\G_{A_n})\right)\\
&\qquad\oplus\left(\Z^{\binom{n-2}{k}}\otimes
H_0(\G_{A_1})\otimes\dots\otimes H_0(\G_{A_{n-1}})
\otimes H_0(\G_{A_n})\right)\\
&\qquad\oplus\left(\Z^{\binom{n-2}{k-1}}\otimes
H_1(\G_{A_1})\otimes\dots\otimes H_1(\G_{A_{n-1}})
\otimes H_1(\G_{A_n})\right). 
\end{align*}
Since $\binom{n-2}{k-1}+\binom{n-2}{k}=\binom{n-1}{k}$ and 
$\binom{n-2}{k-2}+\binom{n-2}{k-1}=\binom{n-1}{k-1}$, 
we obtain the desired conclusion. 
\end{proof}

\begin{theorem}\label{prodSFTHK}
Let $\G=\G_{A_1}\times\G_{A_2}\times\dots\times\G_{A_n}$ be 
a product groupoid of SFT groupoids. 
Then, $\G$ satisfies Conjecture \ref{HK}. 
\end{theorem}
\begin{proof}
By Theorem \ref{SFTproperty} (2), 
the SFT groupoid satisfies Conjecture \ref{HK}. 
The $C^*$-algebra associated with the SFT groupoid is 
the so-called Cuntz-Krieger algebra, 
which is known to be nuclear and satisfy the UCT (\cite{CK80Invent}). 
The theorem then follows from Theorem \ref{HKforproduct} by induction. 
\end{proof}

Next, let us consider the AH conjecture 
for $\G=\G_{A_1}\times\G_{A_2}\times\dots\times\G_{A_n}$. 

A matrix $A$ with entries in $\N\cup\{0\}$ is said to be primitive 
if there exists $k\in\N$ such that every entry of $A^k$ is positive. 
The one-sided shift $(X_A,\sigma_A)$ is topologically mixing 
if and only if the matrix $A$ is primitive (\cite[Proposition 4.5.10]{LM}). 

\begin{lemma}
Any irreducible one-sided shift of finite type is 
continuously orbit equivalent to a one-sided shift of finite type 
which is topologically mixing. 
\end{lemma}
\begin{proof}
The proof is by a slight modification of 
the construction given in \cite[Lemma 3.7]{MM14Kyoto}. 
The $N\times N$ matrix $A$ considered there is clearly primitive. 
The equivalence class $u_A$ of $(1,1,\dots,1)\in\Z^N$ is zero in $\BF(A^t)$. 
For a given $u\in\BF(A^t)$, we choose $(c_1,c_2,\dots,c_N)\in(\N\cup\{0\})^N$ 
whose equivalence class in $\BF(A^t)$ equals $u$, 
and then construct a matrix $B$. 
When choosing $(c_1,c_2,\dots,c_N)$, we may assume that 
the greatest common divisor of $c_i{+}1$'s is one, 
by adding a suitable multiple of $(1,1,\dots,1)$. 
Then, the matrix $B$ becomes primitive, 
which completes the proof. 
\end{proof}

\begin{lemma}
Let $\G=\G_{A_1}\times\G_{A_2}\times\dots\times\G_{A_n}$ be 
a product groupoid of SFT groupoids. 
Then, $\G$ has property TR. 
\end{lemma}
\begin{proof}
The proof is by induction on $n$. 
We know that the SFT groupoid has property TR (\cite[Lemma 6.10]{M15crelle}). 
Assume that the lemma is true for $n{-}1$. 
Suppose that 
we are given $\G=\G_{A_1}\times\G_{A_2}\times\dots\times\G_{A_n}$. 
By the lemma above, we may assume that the matrix $A_n$ is primitive. 
The induction hypothesis implies that 
$\G_{A_1}\times\G_{A_2}\times\dots\times\G_{A_{n-1}}$ has property TR. 
Let $c:\G_{A_n}\to\Z$ be the continuous homomorphism 
defined by $c(x,m,y)=m$. 
Since $A_n$ is primitive, $\H=\Ker c$ is a minimal AF groupoid. 
Therefore, 
$\G_{A_1}\times\G_{A_2}\times\dots\times\G_{A_{n-1}}\times\H$ is 
purely infinite, minimal, and has property TR. 
The kernel of the homomorphism 
\[
\G=\G_{A_1}\times\dots\times\G_{A_n}\longrightarrow
\G_{A_n}\stackrel{c}{\longrightarrow}\Z
\]
equals $\G_{A_1}\times\G_{A_2}\times\dots\times\G_{A_{n-1}}\times\H$. 
Hence, Proposition \ref{GhasTR} applies and yields the conclusion. 
\end{proof}

This lemma, together with Theorem \ref{pi>AH}, implies the following. 

\begin{theorem}\label{prodSFTAH}
Let $\G=\G_{A_1}\times\G_{A_2}\times\dots\times\G_{A_n}$ be 
a product groupoid of SFT groupoids. 
Then, $\G$ satisfies the AH conjecture (Conjecture \ref{AH}). 
\end{theorem}

%%%%%%%%%%%%%%%%%%%%%%%%%%%%%%%%%%%%%%%%%%%%%%%%%%%%%%%%%%%%
\subsection{Classification}

In this subsection, we determine 
when two product groupoids of SFT groupoids are isomorphic to each other 
(Theorem \ref{prodSFTclassify}). 

Let $\G$ be an \'etale groupoid. 
A closed subset $Y\subset\G^{(0)}$ is called $\G$-\'etale 
if the reduction $\G|Y$ is an \'etale groupoid 
with the relative topology from $\G$. 
The following lemmas are easy to prove and is left for the reader. 

\begin{lemma}
Let $\G$ be an \'etale groupoid whose unit space is a Cantor set and 
let $Y\subset\G^{(0)}$ be a closed subset. 
The following are equivalent. 
\begin{enumerate}
\item $Y$ is $\G$-\'etale. 
\item For any $g\in\G|Y$, there exists a compact open $\G$-set $U$ 
such that $g\in U$ and $(r|U)^{-1}(Y)=(s|U)^{-1}(Y)$. 
\end{enumerate}
\end{lemma}

\begin{lemma}\label{Getale}
Let $\G$ be an \'etale groupoid whose unit space is a Cantor set and 
let $U\subset\G$ be a compact open $\G$-set. 
Then $Y=r(U\cap\G')$ is $\G$-\'etale. 
\end{lemma}

In \cite[Section 3]{MM14Kyoto}, 
we introduced the notion of attracting elements. 
Here we weakened its definition a little bit, 
namely we do not require that 
the limit set of the dynamics is a singleton. 

\begin{definition}\label{defofattract}
Let $\G$ be an \'etale groupoid whose unit space is a Cantor set and 
let $g\in\G'$, i.e. $r(g)=s(g)$. 
We say that $g$ is an attracting element 
if there exists a compact open $\G$-set $U$ 
such that $g\in U$ and $r(U)$ is a proper subset of $s(U)$. 
For such $U$, we call the closed set 
\[
Y=\bigcap_{n=1}^\infty r(U^n)
\]
a limit set of $g$. 
Notice that $r(g)$ is contained in every limit set of $g$. 
\end{definition}

For any SFT groupoid $\G_A$ and $x\in X_A$, 
the isotropy group $(\G_A)_x$ is either $0$ or $\Z$, 
and $(\G_A)_x\cong\Z$ if and only if $x\in X_A$ is eventually periodic 
(see \cite[Lemma 3.3]{MM14Kyoto}). 
When $(\G_A)_x\cong\Z$ and $g\in(\G_A)_x$ is attracting, 
there exists a compact open $\G$-set $U$ 
such that $\{x\}$ is the limit set and $U\cap\G'=\{g\}$. 

\begin{theorem}\label{prodSFTclassify}
Let $\G=\G_{A_1}\times\G_{A_2}\times\dots\times\G_{A_m}$ and 
$\H=\G_{B_1}\times\G_{B_2}\times\dots\times\G_{B_n}$ be 
product groupoids of SFT groupoids. 
Then $\G\cong\H$ if and only if the following are satisfied. 
\begin{enumerate}
\item $m=n$. 
\item There exist a permutation $\sigma$ of $\{1,2,\dots,n\}$ and 
isomorphisms $\phi_i:\BF(A_i^t)\to\BF(B_{\sigma(i)}^t)$ 
such that $\det(\id-A_i)=\det(\id-B_{\sigma(i)})$ and 
\[
(\phi_1\otimes\phi_2\otimes\dots\otimes\phi_n)
(u_{A_1}\otimes u_{A_2}\otimes\dots\otimes u_{A_n})
=u_{B_{\sigma(1)}}\otimes u_{B_{\sigma(2)}}\otimes\dots
\otimes u_{B_{\sigma(n)}}. 
\]
In particular, 
$\G_{A_i}$ and $\G_{B_{\sigma(i)}}$ are Morita equivalent. 
\end{enumerate}
\end{theorem}
\begin{proof}
Let us prove the `if' part. 
By permuting indices, we may assume that $\sigma$ is the identity. 
Choose a clopen set $Y_i\subset X_{B_i}$ 
so that $[1_{Y_i}]=\phi_i(u_{A_i})$ in $H_0(\G_{B_i})$ 
(\cite[Lemma 5.3]{M15crelle}). 
It follows from \cite[Theorem 6.2]{M15crelle} that 
$\G_{A_i}$ is isomorphic to $\G_{B_i}|Y_i$. 
Hence $\G$ is isomorphic to 
\[
(\G_{B_1}|Y_1)\times(\G_{B_2}|Y_2)\times\dots\times(\G_{B_n}|Y_n)
\]
which is canonically identified with 
\[
(\G_{B_1}\times\G_{B_2}\times\dots\times\G_{B_n})
|(Y_1\times Y_2\times\dots\times Y_n)
=\H|(Y_1\times Y_2\times\dots\times Y_n). 
\]
Since 
\[
[1_{Y_1\times Y_2\times\dots\times Y_n}]
=\phi_1(u_{A_1})\otimes\phi_2(u_{A_2})\otimes\dots
\otimes\phi_n(u_{A_n})
=u_{B_1}\otimes u_{B_2}\otimes\dots\otimes u_{B_n}
\]
in $H_0(\H)$, 
Theorem \ref{pi} (3) tells us that 
$\H|(Y_1\times Y_2\times\dots\times Y_n)$ is isomorphic to $\H$. 
Therefore we get $\G\cong\H$. 

Let us prove the `only if' part. 
Suppose that $\pi:\G\to\H$ is an isomorphism. 
For 
\[
x=(x_1,x_2,\dots,x_m)\in X_{A_1}\times X_{A_2}\times\dots\times X_{A_m}
=\G^{(0)}, 
\]
the isotropy group $\G_x$ is isomorphic to $\Z^k$, 
where $k=\#\{i\mid\text{$x_i$ is eventually periodic}\}$. 
For any $x\in\G$, we have $\pi(\G_x)=\H_{\pi(x)}$, 
and so $m$ must be equal to $n$. 
Take a point $x=(x_1,x_2,\dots,x_n)\in\G^{(0)}$ such that $\G_x\cong\Z^n$. 
Let $y=(y_1,y_2,\dots,y_n)=\pi(x)$. 
We have $\H_y\cong\Z^n$. 
Let $g_i\in(\G_{A_i})_{x_i}$ be the unique attracting generator 
of $(\G_{A_i})_{x_i}\cong\Z$. 
Define $\tilde g_i\in\G_x$ by 
\[
\tilde g_i=(x_1,\dots,x_{i-1},g_i,x_{i+1},\dots,x_n). 
\]
Analogously we define $h_i\in(\G_{B_i})_{y_i}$ and $\tilde h_i\in\H_y$. 
Consider the subset 
\[
P=\{g\in\G_x\mid\text{$g$ is attracting}\}
\]
of $\G_x\cong\Z$. 
Evidently one has 
\[
\pi(P)=\{h\in\H_y\mid\text{$h$ is attracting}\}. 
\]
Clearly 
$P$ is generated by $\tilde g_1,\tilde g_2,\dots,\tilde g_n$ as a semigroup, 
that is, 
\[
P=\left\{\tilde g_1^{k_1}\tilde g_2^{k_2}\dots\tilde g_n^{k_n}
\mid k_1,k_2,\dots,k_n\in\N\cup\{0\}\right\}. 
\]
Moreover, such a set of generators is unique. 
The same is also true 
for $\{\tilde h_1,\tilde h_2,\dots,\tilde h_n\}\subset\pi(P)$. 
Hence we obtain 
$\{\pi(\tilde g_1),\pi(\tilde g_2),\dots,\pi(\tilde g_n)\}
=\{\tilde h_1,\tilde h_2,\dots,\tilde h_n\}$. 
To simplify notation, by permuting indices, 
we may assume $\pi(\tilde g_i)=\tilde h_i$ for every $i=1,2,\dots,n$. 

Since $h_i\in\G_{B_i}$ is attracting, 
we can find a compact open $\G_{B_i}$-set $V_i\subset\G_{B_i}$ 
such that $h_i$ is in $V_i$, $r(V_i)$ is a proper subset of $s(V_i)$ and 
\[
\bigcap_{k=1}^\infty r(V_i^k)=\{y_i\}. 
\]
Let $D=s(V_1)\times s(V_2)\times\dots\times s(V_n)$. 
For each $i=1,2,\dots,n$, 
we define a compact open $\H$-set $\widetilde V_i\subset\H$ by 
\[
\widetilde V_i=V_1\times\dots\times V_{i-1}\times s(V_i)
\times V_{i+1}\times\dots\times V_n. 
\]
One has 
\[
s(\widetilde V_i)=D,\quad 
\tilde h_1\dots\tilde h_{i-1}\tilde h_{i+1}\dots\tilde h_n\in\widetilde V_i
\]
and 
\[
\bigcap_{k=1}^\infty r(\widetilde V_i^k)=r(\widetilde V_i\cap\H')
=\{y_1\}\times\dots\times\{y_{i-1}\}\times s(V_i)
\times\{y_{i+1}\}\times\{y_n\}. 
\]
By Lemma \ref{Getale}, $r(\widetilde V_i\cap\H')$ is $\H$-\'etale. 
It is easy to see that 
the reduction $\H|r(\widetilde V_i\cap\H')$ is isomorphic to 
$\Z^{n-1}\times(\G_{B_i}|s(V_i))$, 
and its essentially principal part (see Section 2.1) is $\G_{B_i}|s(V_i)$. 

In the same way, 
since $g_i\in\G_{A_i}$ is attracting, 
we can find a compact open $\G_{A_i}$-set $U_i\subset\G_{A_i}$ 
such that $g_i$ is in $U_i$, $r(U_i)$ is a proper subset of $s(U_i)$ and 
\[
\bigcap_{k=1}^\infty r(U_i^k)=\{x_i\}. 
\]
Let $C=s(U_1)\times s(U_2)\times\dots\times s(U_n)$. 
For each $i=1,2,\dots,n$, 
we define a compact open $\G$-set $\widetilde U_i\subset\G$ by 
\[
\widetilde U_i=U_1\times\dots\times U_{i-1}\times s(U_i)
\times U_{i+1}\times\dots\times U_n. 
\]
One has 
\begin{equation}
s(\widetilde U_i)=C,\quad 
\tilde g_1\dots\tilde g_{i-1}\tilde g_{i+1}\dots\tilde g_n\in\widetilde U_i
\label{hare}
\end{equation}
and 
\begin{equation}
\bigcap_{k=1}^\infty r(\widetilde U_i^k)=r(\widetilde U_i\cap\G')
=\{x_1\}\times\dots\times\{x_{i-1}\}\times s(U_i)
\times\{x_{i+1}\}\times\{x_n\}. 
\label{dragon}
\end{equation}
By Lemma \ref{Getale}, $r(\widetilde U_i\cap\G')$ is $\G$-\'etale. 
It is easy to see that 
the reduction $\G|r(\widetilde U_i\cap\G')$ is isomorphic to 
$\Z^{n-1}\times(\G_{A_i}|s(U_i))$, 
and its essentially principal part is $\G_{A_i}|s(U_i)$. 

It follows from \eqref{hare} and \eqref{dragon} that 
$\pi(\widetilde U_i)$ is a compact open $\H$-set satisfying 
\[
s(\pi(\widetilde U_i))=\pi(C),\quad 
\tilde h_1\dots\tilde h_{i-1}\tilde h_{i+1}\dots\tilde h_n
\in\pi(\widetilde U_i)
\]
and 
\[
\bigcap_{k=1}^\infty r(\pi(\widetilde U_i)^k)
=r(\pi(\widetilde U_i)\cap\H'). 
\]
By replacing each $U_i$ with a smaller subset if necessary, 
we may assume $\pi(\widetilde U_i)\subset\widetilde V_i$. 
Furthermore, we may also assume $[1_{s(U_i)}]=u_{A_i}$ in $H_0(\G_{A_i})$. 
We can find a clopen subset $Y_i\subset s(V_i)$ 
such that 
\[
r(\pi(\widetilde U_i)\cap\H')
=\{y_1\}\times\dots\times\{y_{i-1}\}\times Y_i
\times\{y_{i+1}\}\times\{y_n\}. 
\]
The isomorphism $\pi:\G\to\H$ induces an isomorphism 
between $\G|r(\widetilde U_i\cap\G')$ 
and $\H|r(\pi(\widetilde U_i)\cap\H')$, 
and hence induces an isomorphism $\pi_i$ 
between their essentially principal parts 
$\G_{A_i}|s(U_i)$ and $\G_{B_i}|Y_i$. 
Therefore we obtain $\BF(A_i^t)\cong\BF(B_i^t)$ and 
$\det(\id-A_i)=\det(\id-B_i)$ by Corollary \ref{SFTMorita}. 

Let $\phi_i:\BF(A_i^t)\to\BF(B_i^t)$ be the isomorphism 
induced by the isomorphism $\pi_i:\G_{A_i}|s(U_i)\to\G_{B_i}|Y_i$. 
One has $\phi_i(u_{A_i})=\phi_i([1_{s(U_i)}])=[1_{Y_i}]$. 
We would like to show $\pi(C)=Y_1\times Y_2\times\dots\times Y_n$. 
For any $z=(z_1,z_2,\dots,z_n)\in C$, it is easy to see that 
\[
\lim_{k\to\infty}\theta(\widetilde U_i)^k(z)
=(x_1,\dots,x_{i-1},z_i,x_{i+1},\dots,x_n)
\]
holds. 
Similarly, for any $w=(w_1,w_2,\dots,w_n)\in D$, 
\[
\lim_{k\to\infty}\theta(\widetilde V_i)^k(w)
=(y_1,\dots,y_{i-1},w_i,y_{i+1},\dots,y_n)
\]
holds. 
We note that $\pi(\widetilde U_i)$ is contained in $\widetilde V_i$ and 
that $\pi(C)$ is contained in $D$. 
Hence, for $z=(z_1,z_2,\dots,z_n)\in C$, 
\begin{align*}
&\lim_{k\to\infty}\theta(\widetilde V_i)^k(\pi(z))\\
&=\lim_{k\to\infty}
\theta(\pi(\widetilde U_i))^k(\pi(z))\\
&=\pi\left(\lim_{k\to\infty}\theta(\widetilde U_i)^k(z)\right)\\
&=\pi(x_1,\dots,x_{i-1},z_i,x_{i+1},\dots,x_n)\\
&=(y_1,\dots,y_{i-1},\pi_i(z_i),y_{i+1},\dots,y_n), 
\end{align*}
which implies $\pi(z)=(\pi_1(z_1),\pi_2(z_2),\dots,\pi_n(z_n))$. 
Thus, $\pi(C)=Y_1\times Y_2\times\dots\times Y_n$. 
Then we get 
\begin{align*}
&\phi_1(u_{A_1})\otimes\phi_2(u_{A_2})\otimes\dots\otimes\phi_n(u_{A_n})\\
&=[1_{Y_1}]\otimes[1_{Y_2}]\otimes\dots\otimes[1_{Y_n}]\\
&=[1_{Y_1\times Y_2\times\dots\times Y_n}]\\
&=[1_{\pi(C)}]\\
&=H_0(\pi)([1_C])\\
&=H_0(\pi)([1_{s(U_1)}]\otimes[1_{s(U_2)}]\otimes\dots\otimes[1_{s(U_n)}])\\
&=H_0(\pi)(u_{A_1}\otimes u_{A_2}\otimes\dots\otimes u_{A_n})\\
&=H_0(\pi)([1_{\G^{(0)}}])\\
&=[1_{\H^{(0)}}]\\
&=u_{B_1}\otimes u_{B_2}\otimes\dots\otimes u_{B_n}, 
\end{align*}
which completes the proof. 
\end{proof}

\begin{remark}\label{highdimThomp}
Let $(\mathcal{V},\mathcal{E})$ be a directed graph 
consisting of a unique vertex and $k$ edges (i.e.\ loops), where $k\geq2$. 
The associated shift space is called the full shift over $k$ symbols. 
The adjacency matrix of $(\mathcal{V},\mathcal{E})$ is 
the $1\times 1$ matrix $[k]$. 
The groupoid $C^*$-algebra $C^*_r(\G_{[k]})$ of the SFT groupoid $\G_{[k]}$ 
is the Cuntz algebra $\mathcal{O}_k$. 
V.V. Nekrashevych proved that 
the Higman-Thompson group $V_k$ is identified with 
a certain subgroup of the unitary group $U(\mathcal{O}_k)$ 
(see \cite[Proposition 9.6]{Ne04JOP}). 
This identification yields an isomorphism 
between the Higman-Thompson group $V_{k,1}$ and 
the topological full group $[[\G_{[k]}]]$. 
Therefore, when $(X_A,\sigma_A)$ is a shift of finite type 
which is not necessarily a full shift, 
$[[\G_A]]$ can be thought of as a generalization of $V_{k,1}$ 
(see \cite[Remark 6.3]{M15crelle}). 

M. G. Brin introduced 
the notion of higher dimensional Thompson groups $nV_{k,r}$ 
in \cite[Section 4.2]{Brin04GeomD}. 
These groups can be considered as an $n$-dimensional analogue of 
the Higman-Thompson group $V_{k,r}=1V_{k,r}$. 
In \cite{Brin04GeomD, Brin05JAlg}, 
he proved that $V_{k,r}$ and $2V_{2,1}$ are not isomorphic, 
$2V_{2,1}$ is simple and $2V_{2,1}$ is finitely presented. 
He also proved that $nV_{2,1}$ is simple for all $n\in\N$ 
in \cite{Brin10PublMat}. 
In \cite{BL10GeomD}, C. Bleak and D. Lanoue showed that 
$nV_{2,1}\cong n'V_{2,1}$ if and only if $n=n'$. 
W. Dicks and C. Mart\'inez-P\'erez \cite{DMP14Israel} 
generalized this result and proved that 
$nV_{k,r}\cong n'V_{k',r'}$ if and only if 
$n=n'$, $k=k'$ and $\gcd(k{-}1,r)=\gcd(k'{-}1,r')$. 

Define an $r\times r$ matrix $A_{k,r}$ by 
\[
A_{k,r}=\begin{bmatrix}0&0&\ldots&0&k\\1&0&\ldots&0&0\\0&1&\ldots&0&0\\
\vdots&\vdots&\ddots&\vdots&\vdots\\0&0&\ldots&1&0\end{bmatrix}. 
\]
The topological full group $[[\G_{A_{k,r}}]]$ of the SFT groupoid $A_{k,r}$ 
is naturally isomorphic to the Higman-Thompson group $V_{k,r}$ 
(see \cite[Section 6.7.1]{M15crelle}). 
By Theorem \ref{SFTproperty} (2), 
$H_0(\G_{A_{k,r}})\cong\Z_{k{-}1}$, $H_n(\G_{A_{k,r}})=0$ for $n\geq1$ and 
the equivalence class of the characteristic function of $X_{A_{k,r}}$ 
corresponds to $\bar{r}\in\Z_{k{-}1}$. 
It is not so hard to see that 
the higher dimensional Thompson group $nV_{k,r}$ is isomorphic to 
the topological full group $[[\G]]$ of 
the product groupoid 
\[
\G=\G_{A_{k,r}}
\times\overbrace{\G_{A_{k,1}}\times\dots\times\G_{A_{k,1}}}^{n-1}. 
\]
It follows from Theorem \ref{pi} (1) that 
the commutator subgroup $D([[\G]])$ is simple. 
By Proposition \ref{prodSFTH_k}, 
we get $H_l(\G)\cong(\Z_{k{-}1})^{\binom{n{-}1}{l}}$. 
Therefore, Theorem \ref{prodSFTAH} tells us that 
$[[\G]]$ is simple if and only if $k=2$. 
This reproves the result of Brin \cite{Brin10PublMat}. 
Also, by applying Theorem \ref{prodSFTclassify}, 
we get a new proof of the classification theorem 
by Dicks and Mart\'inez-P\'erez \cite{DMP14Israel}. 

Finiteness conditions of (generalized) Higman-Thompson groups are studied 
by many authors. 
K. S. Brown \cite{Br87JPAA} proved that $V_{k,r}$ is of type F$_\infty$. 
By using his method, we proved that 
the topological full group $[[\G_A]]$ of any SFT groupoid $\G_A$ is 
of type F$_\infty$ (\cite[Theorem 6.21]{M15crelle}). 
As for the higher dimensional Thompson groups, 
D. H. Kochloukova, C. Mart\'inez-P\'erez and B. E. A. Nucinkis 
\cite{KMPN13ProcEdinb} 
showed that $2V_{2,1}$ and $3V_{2,1}$ are of type F$_\infty$. 
This was later generalized to all $nV_{2,1}$ 
by M. Fluch, M. Marschler, S. Witzel and M. C. B. Zaremsky \cite{FMWZ13PJM}. 
Recently, C. Mart\'inez-P\'erez, F. Matucci and B. E. A. Nucinkis 
\cite{MPMN13} 
proved that $nV_{k,1}$ (and many other relatives) are of type F$_\infty$. 
We do not know if the same holds for topological full groups 
of product groupoids of SFT groupoids 
$\G=\G_{A_1}\times\G_{A_2}\times\dots\times\G_{A_n}$. 
\end{remark}

%%%%%%%%%%%%%%%%%%%%%%%%%%%%%%%%%%%%%%%%%%%%%%%%%%%%%%%%%%%%
\subsection{An embedding theorem}

%M. Kato proved that 
%the higher dimensional Thompson group $nV_{2,1}$ has the Haagerup property 
%(\cite[Corollary 6.1]{Kato15}). 
%In this subsection, by using this result, 
%we prove that the topological full group of 
%any product groupoid of SFT groupoids has the Haagerup property. 

Let $(X_A,\sigma_A)$ be a one-sided irreducible SFT 
associated with a finite directed graph $(\mathcal{V},\mathcal{E})$ 
($A$ is the adjacency matrix of the graph). 
For every edge $e\in\mathcal{E}$, 
we define the clopen set $C_e\subset X_A$ by 
\[
C_e=\{(x_k)_{k\in\N}\in X_A\mid x_1=e\}, 
\]
so that $\{C_e\mid e\in\mathcal{E}\}$ is a clopen partition of $X_A$. 
Define a compact open $\G_A$-set $U_e\subset\G_A$ by 
\[
U_e=\{(x,1,y)\in\G_A\mid x\in C_e,\ \sigma_A(x)=y\}. 
\]
One has 
\[
r(U_e)=C_e,\quad s(U_e)=\bigcup_{t(e)=i(f)}C_f. 
\]
Since 
\[
1_{U_e}1_{U_e}^*=1_{r(U_e)}\quad\text{and}\quad 
1_{U_e}^*1_{U_e}=1_{s(U_e)}
\]
in $C^*_r(\G_A)$, we get 
\[
\sum_{e\in\mathcal{E}}1_{U_e}1_{U_e}^*=1_{X_A}\quad\text{and}\quad 
1_{U_e}^*1_{U_e}=\sum_{t(e)=i(f)}1_{U_f}1_{U_f}^*, 
\]
which are called the Cuntz-Krieger relations. 
The Cuntz-Krieger algebra $C^*_r(\G_A)$ is characterized 
as the universal $C^*$-algebra 
generated by partial isometries $\{1_{U_e}\mid e\in\mathcal{E}\}$ 
subject to the Cuntz-Krieger relations (\cite{CK80Invent}). 
We remark that 
the Cuntz-Krieger algebra $C^*_r(\G_A)$ is simple and purely infinite. 

\begin{proposition}
Let $\G_A$ be an SFT groupoid and 
let $\H$ be a minimal, purely infinite \'etale groupoid. 
Suppose that $\phi:H_0(\G_A)\to H_0(\H)$ is a homomorphism 
satisfying $\phi([1_{\G_A^{(0)}}])=[1_{\H^{(0)}}]$. 
Then there exists a unital homomorphism 
$\pi:C^*_r(\G_A)\to C^*_r(\H)$ such that the following hold. 
\begin{enumerate}
\item $\pi(C(\G_A^{(0)}))\subset C(\H^{(0)})$. 
\item For any compact open $\G_A$-set $U$, 
there exists a compact open $\H$-set $V$ 
such that $\pi(1_U)=1_V$. 
\item For any clopen set $C\subset\G_A^{(0)}$, 
$[\pi(1_C)]=\phi([1_C])$ in $H_0(\H)$. 
\end{enumerate}
In particular, 
$\pi$ induces an embedding of $[[\G_A]]$ into $[[\H]]$. 
\end{proposition}
\begin{proof}
Let $(\mathcal{V},\mathcal{E})$ be a finite directed graph 
whose adjacency matrix is $A$. 
For $e\in\mathcal{E}$, 
let $C_e\subset X_A$ and $U_e\subset\G_A$ as above. 
By Theorem \ref{SFTproperty} (2), 
the elements $[1_{C_e}]$ generate $H_0(\G_A)$. 
We have 
\[
\sum_{e\in\mathcal{E}}\phi([1_{C_e}])=\phi([1_{X_A}])
=[1_{\H^{(0)}}]. 
\]
It follows from \cite[Lemma 5.3]{M15crelle} that 
there exists a clopen partition 
$\{D_e\mid e\in\mathcal{E}\}$ of $\H^{(0)}$ 
such that $[1_{D_e}]=\phi([1_{C_e}])$ in $H_0(\H)$. 
Since 
\[
[1_{D_e}]=\phi([1_{C_e}])=\sum_{t(e)=i(f)}\phi([1_{C_f}])
=\sum_{t(e)=i(f)}[1_{D_f}], 
\]
by means of Theorem \ref{pi} (3), 
we can find a compact open $\H$-set $V_e\subset\H$ satisfying 
\[
r(V_e)=D_e,\quad s(V_e)=\bigcup_{t(e)=i(f)}D_f. 
\]
Hence we obtain the Cuntz-Krieger relations 
\[
\sum_{e\in\mathcal{E}}1_{V_e}1_{V_e}^*=1_{\H^{(0)}}\quad\text{and}\quad 
1_{V_e}^*1_{V_e}=\sum_{t(e)=i(f)}1_{V_f}1_{V_f}^*. 
\]
Therefore 
there exists a unital homomorphism $\pi:C^*_r(\G_A)\to C^*_r(\H)$ 
such that $\pi(1_{U_e})=1_{V_e}$. 

The Cartan subalgebra $C(\G_A^{(0)})=C(X_A)$ is generated 
by projections of the form 
\[
(1_{U_{e_1}}1_{U_{e_2}}\dots1_{U_{e_k}})
(1_{U_{e_1}}1_{U_{e_2}}\dots1_{U_{e_k}})^*, 
\]
which is mapped by $\pi$ to 
\[
(1_{V_{e_1}}1_{V_{e_2}}\dots1_{V_{e_k}})
(1_{V_{e_1}}1_{V_{e_2}}\dots1_{V_{e_k}})^*\in C(\H^{(0)}). 
\]
This proves (1). 

Any compact open $\G_A$-set can be written as a finite disjoint union 
of compact open $\G_A$-sets of the form 
\[
U=(U_{e_1}U_{e_2}\dots U_{e_k})(U_{f_1}U_{f_2}\dots U_{f_l})^{-1}. 
\]
For such $U\subset\G_A$, put 
\[
V=(V_{e_1}V_{e_2}\dots V_{e_k})(V_{f_1}V_{f_2}\dots V_{f_l})^{-1}. 
\]
It is easy to see that the homomorphism $\pi$ sends $1_U$ to $1_V$. 
This proves (2). 

(3) is clear from the construction. 
\end{proof}

\begin{theorem}\label{embedding}
Let $\G=\G_{A_1}\times\G_{A_2}\times\dots\times\G_{A_n}$ be 
a product groupoid of SFT groupoids. 
Let $\H_1,\H_2,\dots,\H_n$ be minimal, purely infinite \'etale groupoids 
and let $\H=\H_1\times\H_2\times\dots\times\H_n$. 
For each $i=1,2,\dots,n$, 
suppose that $\phi_i:H_0(\G_{A_i})\to H_0(\H_i)$ is a homomorphism 
satisfying $\phi_i([1_{\G_{A_i}^{(0)}}])=[1_{\H_i^{(0)}}]$. 
Then there exists a unital homomorphism 
$\pi:C^*_r(\G)\to C^*_r(\H)$ such that the following hold. 
\begin{enumerate}
\item $\pi(C(\G^{(0)}))\subset C(\H^{(0)})$. 
\item For any compact open $\G$-set $U$, 
there exists a compact open $\H$-set $V$ 
such that $\pi(1_U)=1_V$. 
\item For any clopen set $C\subset\G^{(0)}$, 
$[\pi(1_C)]=\phi([1_C])$ in $H_0(\H)$, 
where $\phi=\phi_1\otimes\phi_2\otimes\dots\otimes\phi_n$. 
\end{enumerate}
In particular, 
$\pi$ induces an embedding of $[[\G]]$ into $[[\H]]$. 
\end{theorem}
\begin{proof}
By the proposition above, 
we can construct unital homomorphisms $\pi_i:C^*_r(\G_{A_i})\to C^*_r(\H_i)$ 
satisfying the properties described there. 
Then the unital homomorphism $\pi=\pi_1\otimes\pi_2\otimes\dots\otimes\pi_n$ 
from $C^*_r(\G)$ to $C^*_r(\H)$ meets the requirement. 
\end{proof}

\begin{corollary}
Let $\G=\G_{A_1}\times\G_{A_2}\times\dots\times\G_{A_n}$ be 
a product groupoid of SFT groupoids. 
Then the topological full group $[[\G]]$ is embeddable into 
the higher dimensional Thompson group $nV_{2,1}$. 
%In particular, $[[\G]]$ has the Haagerup property. 
\end{corollary}
\begin{proof}
Let $\H$ be the $n$-fold product of $\G_{[2]}$. 
As discussed in Remark \ref{highdimThomp}, 
$[[\H]]$ is canonically isomorphic to $nV_{2,1}$. 
It follows from the theorem above that 
$[[\G]]$ is embeddable into $nV_{2,1}$. 
\end{proof}

%%%%%%%%%%%%%%%%%%%%%%%%%%%%%%%%%%%%%%%%%%%%%%%%%%%%%%%%%%%%
\subsection{Abelianization}

In this subsection, 
we would like to determine the abelianization $[[\G]]_\ab$ 
of the topological full group $[[\G]]$ 
for product groupoids $\G=\G_{A_1}\times\G_{A_2}\times\dots\times\G_{A_n}$ 
of SFT groupoids. 
By Theorem \ref{prodSFTAH}, we have already known that 
$\G$ satisfies the AH conjecture, i.e. 
\[
\begin{CD}
H_0(\G)\otimes\Z_2@>j>>[[\G]]_\ab@>I>>H_1(\G)@>>>0
\end{CD}
\]
is exact. 
We will show that $\G$ has the strong AH property when $n=2$, 
and that $\G$ may fail to have the strong AH property when $n\geq3$. 

Let $n\geq2$ be a natural number and 
let $k:\{1,2,\dots,n\}\to\N\setminus\{1\}$ be a map. 
First, we would like to begin with the groupoid 
$\G=\G_{[k(1)]}\times\G_{[k(2)]}\times\dots\times\G_{[k(n)]}$ 
($[k(d)]$ is a $1\times1$ matrix). 
By Theorem \ref{SFTproperty} (2) and Proposition \ref{prodSFTH_k}, we have 
\[
H_0(\G)\cong\Z_g\quad\text{and}\quad H_1(\G)\cong(\Z_g)^{n-1}, 
\]
where $g=\gcd(k(1){-}1,k(2){-}1,\dots,k(n){-}1)$. 
In later discussion, we actually need to work with the following three cases. 
\begin{description}
\item[Case 0]$k$ is constant and $k(1)=k(2)=\dots=k(n)=l$ 
for $l\in\N\setminus\{1\}$. 
\item[Case 1]$k(1)=3$ and $k(2)=\dots=k(n)=5$. 
\item[Case 2]$k(1)=k(2)=3$ and $k(3)=\dots=k(n)=5$. 
\end{description}
So, the reader may restrict the attention to these special cases, 
but we proceed with the discussion for arbitrary $k$ for the moment. 
As mentioned in Remark \ref{highdimThomp}, 
the topological full group of $\G_{[2]}\times\G_{[2]}$ is 
the two dimensional Thompson group $2V_{2,1}$. 
Presentations of the group $2V_{2,1}$ was given 
by M. G. Brin \cite{Brin05JAlg}, 
and later it was extended to $nV_{2,1}$ 
by J. Hennig and F. Matucci \cite{HM12PJM}. 
We have to generalize some arguments of these works 
to $\G=\G_{[k(1)]}\times\G_{[k(2)]}\times\dots\times\G_{[k(n)]}$. 

Fix $n\geq2$ and $k:\{1,2,\dots,n\}\to\N\setminus\{1\}$. 
Let $X_{[k(d)]}=\{0,1,\dots,k(d){-}1\}^\N$ for $d=1,2,\dots,n$. 
With the product topology, $X_{[k(d)]}$ is the Cantor set. 
The shift map $\sigma_{[k(d)]}:X_{[k(d)]}\to X_{[k(d)]}$ is given by 
\[
\sigma_{[k(d)]}(x)_m=x_{m+1}\quad m\in\N,\ x=(x_m)_m\in X_{[k(d)]}. 
\]
The dynamical system $(X_{[k(d)]},\sigma_{[k(d)]})$ is 
the full shift over $k(d)$ symbols. 
Let $Z_{n,k}$ be the product space of $X_{[k(d)]}$, i.e. 
\[
Z_{n,k}=X_{[k(1)]}\times X_{[k(2)]}\times\dots\times X_{[k(n)]}. 
\]
For $d\in\{1,2,\dots,n\}$, 
we define $\sigma_d:Z_{n,k}\to Z_{n,k}$ by 
\[
\sigma_d=\id\times\dots\times
\id\times\sigma_{[k(d)]}\times\id\times\dots\times\id,
\]
where $\sigma_{[k(d)]}$ acts on the $d$-th coordinate. 
For all $i\in\N$ and $d\in\{1,2,\dots,n\}$, 
we define homeomorphisms 
$s_{i,d},\tau_i\in\Homeo(Z_{n,k}\times\N)$ by 
\[
s_{i,d}(z,j)
=\begin{cases}(z,j)&j<i\\
\left(\sigma_d(z),j+z^{(d)}_1\right)&j=i\\
(z,j{+}k(d))&j>i, \end{cases}
\]
and 
\[
\tau_i(z,j)=\begin{cases}(z,j)&j\neq i,i{+}1\\
(z,i{+}1)&j=i\\(z,i)&j=i{+}1, \end{cases}
\]
for $z=(z^{(1)},z^{(2)},\dots,z^{(n)})\in Z_{n,k}$ and $j\in\N$, 
where $z^{(d)}_1$ denotes the first coordinate of $z^{(d)}\in X_{[k(d)]}$.
Set $\mathcal{S}_{n,k}=\{s_{i,d},\tau_i\mid i\in\N,\ d=1,2,\dots,n\}$. 
Let $W_{n,k}\subset\Homeo(Z_{n,k}\times\N)$ denote 
the subgroup generated by $\mathcal{S}_{n,k}$. 
It is easy to see that 
these generators satisfy the following set $\mathcal{R}_{n,k}$ of relations. 
\[
\mathcal{R}_{n,k}\ \begin{cases}
s_{i,d}s_{j,d'}=s_{j+k(d)-1,d'}s_{i,d}&\quad i<j,\ 1\leq d,d'\leq n \\
\tau_i^2=1&\quad i\in\N, \\
\tau_i\tau_j=\tau_j\tau_i&\quad \lvert i-j\rvert\geq2, \\
\tau_i\tau_{i+1}\tau_i=\tau_{i+1}\tau_i\tau_{i+1}&\quad i\in\N, \\
s_{i,d}\tau_j=\tau_{j+k(d)-1}s_{i,d}&\quad i<j, \\
s_{i,d}\tau_i=\tilde\tau_{i,d}s_{i+1,d},&\quad i\in\N, \\
s_{i,d}\tau_j=\tau_js_{i,d}&\quad i>j+1, \\
s_{i,d'}s_{i+1,d'}\dots s_{i+k(d)-1,d'}s_{i,d}& \\
\quad=\alpha_{i,d,d'}s_{i,d}s_{i+1,d}\dots s_{i+k(d')-1,d}s_{i,d'}&\quad 
i\in\N,\ d\neq d', 
\end{cases}
\]
where $\tilde\tau_{i,d}$ is $\tau_{i+k(d)-1}\dots\tau_{i+1}\tau_i$ and 
$\alpha_{i,d,d'}$ is any word in $\{\tau_i\mid i\in\N\}$ satisfying 
\[
\alpha_{i,d,d'}(z,j)=\begin{cases}(z,i{+}qk(d'){+}p)&j=i{+}pk(d){+}q\quad
\text{for some }0{\leq}p{<}k(d'),\ 0{\leq}q{<}k(d)\\
(z,j)&\text{else. }\end{cases}
\]
Clearly one has 
\begin{itemize}
\item the permutation on $\N$ induced by $\tilde\tau_{i,d}$ is odd 
if and only if $k(d)$ is odd, 
\item when $k(d)=k(d')=l$, 
the permutation on $\N$ induced by $\alpha_{i,d,d'}$ is odd 
if and only if $l\in4\Z+2$ or $l\in4\Z+3$, 
\item when $\{k(d),k(d')\}=\{3,5\}$, 
the permutation on $\N$ induced by $\alpha_{i,d,d'}$ is even. 
\end{itemize}

\begin{proposition}
[{\cite[Theorem 2]{Brin05JAlg}},{\cite[Theorem 1.1]{HM12PJM}}]
The group $W_{n,k}$ is presented 
by using the generators $\mathcal{S}_{n,k}$ and 
the relations $\mathcal{R}_{n,k}$. 
\end{proposition}
\begin{proof}
We can prove this proposition exactly in the same way 
as \cite{Brin05JAlg} and \cite{HM12PJM}, 
using $k(d)$-ary trees instead of binary trees. 
\end{proof}

In the same way as $nV_{2,1}$, 
the topological full group $[[\G]]$ of 
\[
\G=\G_{[A_{k(1),r}]}\times\G_{[k(2)]}\times\dots\times\G_{[k(n)]}
\]
is canonically isomorphic to the subgroup of $W_{n,k}$ 
consisting of elements $\gamma$ 
satisfying $\gamma(z,j)=(z,j)$ unless $j\in\{1,2,\dots,r\}$ 
(see Remark \ref{highdimThomp} for the definition of the matrix $A_{l,r}$). 
Especially, in Case 0, 
the subgroup of $W_{n,k}$ consisting of elements $\gamma$ 
satisfying $\gamma(z,j)=(z,j)$ unless $j\in\{1,2,\dots,r\}$ is 
canonically isomorphic to the higher-dimensional Thompson group $nV_{l,r}$. 
In what follows, 
we identify the topological full group $[[\G]]$ 
with this subgroup of $W_{n,k}$. 
It is easy to see that 
$s_{1,d'}^{-1}s_{1,d}$ and $s_{1,d}^{-1}\tau_1s_{1,d}$ are in $[[\G]]$. 
The element $s_{1,d}^{-1}\tau_1s_{1,d}$ is the transposition 
corresponding to the generator of $H_0(\G_{[k(d)]})\cong\Z_{k(d)-1}$. 
For $d\neq d'$, 
the element $s_{1,d'}^{-1}s_{1,d}\in[[\G]]$ is called the baker's map. 
When $n=2$ and $k(1)=k(2)=l$, 
as observed in \cite[Lemma 8.1]{Brin04GeomD}, 
the homeomorphism $s_{1,2}^{-1}s_{1,1}$ on $2V_{l,1}$ is 
conjugate to the two-sided full shift over $l$ symbols. 

\begin{remark}
The groupoid $C^*$-algebra $C^*_r(\G_{[l]})$ of $\G_{[l]}$ is 
the Cuntz algebra, 
which is generated by $l$ isometries $t_1,t_2,\dots,t_l$ satisfying 
\[
\sum_{i=1}^lt_it_i^*=1. 
\]
It is well-known that the unitary 
\[
u=\sum_{i=1}^lt_i^*\otimes t_i\in C^*_r(\G_{[l]})\otimes C^*_r(\G_{[l]})
\]
is a generator of $K_1(C^*_r(\G_{[l]})\otimes C^*_r(\G_{[l]}))\cong\Z_{l-1}$. 
The automorphism of $C(X_{[l]}\times X_{[l]})$ induced by $u$ 
corresponds to the baker's map $s_{1,2}^{-1}s_{1,1}\in 2V_{l,1}$. 
\end{remark}

\begin{lemma}\label{tauisnonzero}
Let $\G=\G_{[A_{k(1),r}]}\times\G_{[k(2)]}\times\dots\times\G_{[k(n)]}$. 
\begin{enumerate}
\item In Case 0, suppose that $k(1)=\dots=k(n)=l$ is in $4\Z+1$. 
Then, there exists a homomorphism $\rho:[[\G]]\to\Z_2$ 
such that $\rho(s_{1,d'}^{-1}s_{1,d})=\bar{0}$ and 
$\rho(s_{1,d'}^{-1}\tau_1s_{1,d})=\bar{1}$ for any $d,d'$. 
\item In Case 1, 
there exists a homomorphism $\rho:[[\G]]\to\Z_2$ 
such that $\rho(s_{1,d'}^{-1}s_{1,d})=\bar{0}$ and 
$\rho(s_{1,d'}^{-1}\tau_1s_{1,d})=\bar{1}$ for any $d,d'$. 
\item In Case 2, 
there exists a homomorphism $\rho:[[\G]]\to\Z_4$ 
such that $\rho(s_{1,d}^{-1}s_{1,1})=\bar{1}$ for any $d\neq1$ and 
$\rho(s_{1,d}^{-1}\tau_1s_{1,d})=\overline{2}$ for any $d$. 
\end{enumerate}
\end{lemma}
\begin{proof}
(1)(2)
By the proposition above, 
we can find a homomorphism $\rho:W_{n,k}\to\Z_2$ 
such that $\rho(s_{i,d})=\bar{0}$ and $\rho(\tau_i)=\bar{1}$. 

(3)
By the proposition above, 
we can find a homomorphism $\rho:W_{n,k}\to\Z_4$ 
such that $\rho(s_{i,1})=\bar{1}$, $\rho(s_{i,d})=\bar{0}$ for any $d\neq1$ 
and $\rho(\tau_i)=\overline{2}$. 
\end{proof}

Let $\G$ be the product groupoid of $n$ copies of $\G_{[l]}$. 
Then, $[[\G]]$ is isomorphic to $nV_{l,1}$. 
When $l$ is even, $H_0(\G)\otimes\Z_2\cong\Z_{l-1}\otimes\Z_2$ is trivial, 
and so $[[\G]]_\ab$ is isomorphic to $H_1(\G)\cong(\Z_{l-1})^{n-1}$ 
via the index map. 
When $\l$ is in $4\Z+1$, by Lemma \ref{tauisnonzero} (1), 
the homomorphism $j:H_0(\G)\otimes\Z_2\to[[\G]]_\ab$ is injective. 
Thus, $\G$ has the strong AH property. 
Moreover, Lemma \ref{tauisnonzero} (1) tells us that the extension is trivial. 
Hence we obtain $[[\G]]_\ab\cong(\Z_{l-1})^{n-1}\oplus\Z_2$. 
When $n=2$ and $l\in4\Z+3$, in the same way as Lemma \ref{tauisnonzero} (3), 
one can show that there exists a surjective homomorphism $[[\G]]\to\Z_{2l-2}$. 
Therefore, $\G$ has the strong AH property, 
and the extension does not split. 
When $n\geq3$ and $l\in4\Z+3$, by virtue of Lemma \ref{tauiszero} below, 
we can conclude that 
the homomorphism $j:H_0(\G)\otimes\Z_2\to[[\G]]_\ab$ vanishes. 
In other words, $\G$ does not have the strong AH property and 
$[[\G]]_\ab\cong(\Z_{l-1})^{n-1}$. 
These arguments are summarized as follows. 
This can be regarded as a special case of Theorem \ref{prodSFTAbel}. 

\begin{theorem}\label{highdimThompAbel}
For $n\geq1$, $k\geq2$ and $r\geq1$, the following hold. 
\begin{enumerate}
\item When $n=1$, 
\[
(1V_{k,r})_\ab\cong\begin{cases}\Z_2&k\in2\Z\\0&k\in2\Z+1. \end{cases}
\]
\item When $n=2$, 
\[
(2V_{k,r})_\ab\cong\begin{cases}\Z_{k-1}&k\in2\Z\\
\Z_{k-1}\oplus\Z_2&k\in4\Z+1\\
\Z_{2k-2}&k\in4\Z+3. \end{cases}
\]
\item When $n\geq3$, 
\[
(nV_{k,r})_\ab\cong\begin{cases}(\Z_{k-1})^{n-1}&k\in2\Z\text{ or }k\in4\Z+3\\
(\Z_{k-1})^{n-1}\oplus\Z_2&k\in4\Z+1. \end{cases}
\]
\end{enumerate}
\end{theorem}

We turn, now, to the consideration of the general case. 

In order to state Lemma \ref{baker}, we introduce some notations. 
Let $\G$ and $\H$ be minimal purely infinite groupoids. 
For $\gamma\in[[\G]]$ and a clopen set $F\subset\H^{(0)}$, 
we define $\gamma\odot\id_F\in[[\G\times\H]]$ by 
\[
(\gamma\odot\id_F)(x,y)=\begin{cases}(\gamma(x),y)&y\in F\\
(x,y)&y\notin F. \end{cases}
\]
It is easy to see that $I(\gamma\odot\id_F)=I(\gamma)\otimes[1_F]$, 
where $H_1(\G)\otimes H_0(\H)$ is regarded 
as a subgroup of $H_1(\G\times\H)$. 

Let $g\in\N$. 
Let $a\in H_0(\G)$ and $b\in H_0(\H)$ be such that $ga=0$ and $gb=0$. 
Choose nonempty clopen sets $E\subset\G^{(0)}$ and $F\subset\H^{(0)}$ so that 
$E\neq\G^{(0)}$, $F\neq\H^{(0)}$, $[1_E]=a$ and $[1_F]=b$. 
There exist compact open $G$-sets $U_0,U_1,\dots,U_g$ such that 
$s(U_i)=E$ for every $i=0,1,\dots,g$ and 
$\{r(U_i)\mid i=0,1,\dots,g\}$ is a partition of $E$. 
In the same way, we can find compact open $\H$-sets $V_0,V_1,\dots,V_g$ 
such that $s(V_i)=F$ for every $i=0,1,\dots,g$ and 
$\{r(V_i)\mid i=0,1,\dots,g\}$ is a partition of $F$. 
Then, we define a homeomorphism $\beta\in[[\G\times\H]]$ 
associated with $g,a,b$ by 
\[
\beta(x,y)=\begin{cases}
\theta(U_i\times V_i^{-1})(x,y)&x\in E,\ y\in r(V_i)\\
(x,y)&\text{otherwise. }\end{cases}
\]
This may be thought of 
as a variant of the baker's map $s_{1,2}^{-1}s_{1,1}\in2V_{l,1}$. 

\begin{lemma}\label{baker}
In the setting above, one has the following. 
\begin{enumerate}
\item If two proper nonempty clopen sets $F_1,F_2\subset\H^{(0)}$ satisfy 
$[1_{F_1}]=[1_{F_2}]$ in $H_0(\H)$, then 
$\gamma\odot\id_{F_1}$ and $\gamma\odot\id_{F_2}$ are conjugate 
in $[[\G\times\H]]$. 
\item If $F\subset\H^{(0)}$ is a nonempty clopen set 
with $[1_F]=0$ in $H_0(\H)$, then $\gamma\odot\id_F$ belongs to 
the commutator subgroup $D([[\G\times\H]])$. 
\item $\beta$ is uniquely determined by $g$, $a$ and $b$ 
up to conjugacy in $[[\G\times\H]]$. 
\item The equivalence class of $\beta^g$ in $[[\G\times\H]]_\ab$ is 
equal to the image of 
\[
\frac{g(g+1)}{2}a\otimes b\otimes\bar{1}
\in H_0(\G)\otimes H_0(\H)\otimes\Z_2. 
\]
\end{enumerate}
\end{lemma}
\begin{proof}
(1)
By Theorem \ref{pi} (3), 
there exists $\alpha\in[[\H]]$ such that $\alpha(F_1)=F_2$. 
Then $(\id\times\alpha)(\gamma\odot\id_{F_1})(\id\times\alpha^{-1})$ 
equals $\gamma\odot\id_{F_2}$. 

(2)
Suppose $F\neq\H^{(0)}$. 
We can find disjoint nonempty clopen sets $F_1,F_2$ 
such that $F_1\cup F_2=F$ and $[1_{F_1}]=[1_{F_2}]=0$. 
It follows from (1) that 
$\gamma\odot\id_F$ and $\gamma\odot\id_{F_i}$ are conjugate 
for each $i=1,2$. 
Since $(\gamma\odot\id_{F_1})(\gamma\odot\id_{F_2})=\gamma\odot\id_F$, 
$\gamma\odot\id_F$ is in $D([[\G\times\H]])$, as desired. 
When $F=\H^{(0)}$, 
we can apply the same argument for $F_1,F_2$ as above. 

(3)
This can be proved in the same way as (1). 

(4)
Define a transposition $\tau\in[[\G\times\H]]$ by 
\[
\tau(x,y)=\begin{cases}
\left(\theta(U_j\times V_i)\circ\theta(U_i\times V_j)^{-1}\right)(x,y)
&x\in r(U_i),\ y\in r(V_j),\ i\neq j\\
(x,y)&\text{otherwise. }\end{cases}
\]
It is not so hard to see that 
$\tau\beta$ is a product of $g{+}1$ elements with mutually disjoint support 
and each of them is conjugate to $\beta$. 
This means $\tau\beta^{-g}$ is in $D([[\G\times\H]])$. 
We have 
\[
\sum_{i<j}[1_{r(U_i)}\times1_{r(V_j)}]=\frac{g(g+1)}{2}a\otimes b
\]
in $H_0(\G\times\H)$, which implies the conclusion. 
\end{proof}

\begin{lemma}\label{tauiszero}
Let $k$ be a natural number such that $k\in4\Z{+}3$ and 
let $\G=\G_{[k]}\times\G_{[k]}\times\G_{[k]}$. 
The index map $I:[[\G]]\to H_1(\G)$ induces 
an isomorphism $[[\G]]_\ab\cong H_1(\G)$. 
In particular, $\G$ does not have the strong AH property. 
\end{lemma}
\begin{proof}
By Theorem \ref{prodSFTAH}, 
\[
\begin{CD}
H_0(\G)\otimes\Z_2@>j>>[[\G]]_\ab@>I>>H_1(\G)@>>>0
\end{CD}
\]
is exact. 
Since $k$ is odd, $H_0(\G)\otimes\Z_2\cong\Z_{k-1}\otimes\Z_2=\Z_2$. 
Let $X_{[k]}=\{0,1,\dots,k{-}1\}^\N$ and 
let $(X_{[k]},\sigma_{[k]})$ be the full shift over $k$ symbols. 
As in Section 5.4, 
we define the clopen set $C_i\subset X_{[k]}$ by 
\[
C_i=\{(x_n)_{n\in\N}\in X_{[k]}\mid x_1=i\}. 
\]
Define a compact open $\G_{[k]}$-set $U_i\subset\G_{[k]}$ by 
\[
U_i=\{(x,1,y)\in\G_{[k]}\mid x\in C_i,\ \sigma_{[k]}(x)=y\}. 
\]
Let $t\in[[\G]]_\ab$ be the image of 
the generator of $H_0(\G)\otimes\Z_2\cong\Z_2$. 
We would like to show $t=0$. 
Define $\beta_{12}\in[[\G]]$ by 
\[
\beta_{12}(x,y,z)=
\theta(U_i\times U_i^{-1}\times X_{[k]})(x,y,z)\
\quad\text{when $y\in C_i$}
\]
for $(x,y,z)\in\G^{(0)}=X_{[k]}\times X_{[k]}\times X_{[k]}$. 
The homeomorphism $\beta_{12}$ is the baker's map 
acting on the first and second coordinates of 
$X_{[k]}\times X_{[k]}\times X_{[k]}$. 
By Lemma \ref{baker} (4), 
one sees $(k{-}1)[\beta_{12}]=t$ in $[[\G]]_\ab$. 
We can define $\beta_{23}\in[[\G]]$ in the same way by 
\[
\beta_{23}(x,y,z)=
\theta(X_{[k]}\times U_i\times U_i^{-1})(x,y,z)\
\quad\text{when $z\in C_i$}. 
\]
Again one has $(k{-}1)[\beta_{23}]=t$. 
It is easy to see that 
$\beta_{12}\beta_{23}$ is equal to the baker's map 
acting on the first and third coordinates of 
$X_{[k]}\times X_{[k]}\times X_{[k]}$. 
Therefore, we get $(k{-}1)[\beta_{12}\beta_{23}]=t$. 
Consequently, we obtain $2t=t$, thus $t=0$. 
\end{proof}

Let $\G=\G_{A_1}\times\G_{A_2}\times\dots\times\G_{A_n}$ be 
a product groupoid of SFT groupoids. 
We compute the abelianization $[[\G]]_\ab$. 
By Corollary \ref{H1ofprod}, $H_1(\G)$ is isomorphic to 
\begin{align*}
&\bigoplus_{q_1+q_2+\dots+q_n=1}
H_{q_1}(\G_{A_1})\otimes H_{q_2}(\G_{A_2})\otimes\dots
\otimes H_{q_n}(\G_{A_n})\\
&\quad\oplus\bigoplus_{p=1}^{n-1}
H_0(\G_{A_1})\otimes\dots\otimes H_0(\G_{A_{p-1}})
\otimes\Tor(H_0(\G_{A_p}),H_0(\G_{A_{p+1}}\times\dots\times \G_{A_n})). 
\end{align*}
We write 
\[
T_p=H_0(\G_{A_1})\otimes\dots\otimes H_0(\G_{A_{p-1}})
\otimes\Tor(H_0(\G_{A_p}),H_0(\G_{A_{p+1}}\times\dots\times \G_{A_n})). 
\]
For each $d=1,2,\dots,n$, we assume 
\[
H_0(\G_{A_d})\cong\bigoplus_{i=1}^{h(d)}\Z_{m(d,i)}
\]
for $h(d)\in\N\cup\{0\}$ and $m(d,i)\in\{0,2,3,\dots\}$, 
where $\Z_0$ is understood as $\Z$. 
Set 
\[
J=\{(i_1,i_2,\dots,i_n)\in\N^n\mid1\leq i_d\leq h(d)\}. 
\]
The set $J$ is empty if and only if 
there exists $d$ such that $H_0(\G_{A_d})=0$. 
In such a case one has $H_0(\G)=H_1(\G)=0$ and $[[\G]]=D([[\G]])$. 
For $(i_1,i_2,\dots,i_n)\in J$ and $p\in\{1,2,\dots,n{-}1\}$, 
we write 
\[
S(i_1,i_2,\dots,i_n)
=\Z_{m(1,i_1)}\otimes\Z_{m(2,i_2)}\otimes\dots\otimes\Z_{m(n,i_n)}
\subset H_0(\G)
\]
and 
\begin{align*}
&T_p(i_1,i_2,\dots,i_n)\\
&=\Z_{m(1,i_1)}\otimes\dots\otimes\Z_{m(p-1,i_{p-1})}\otimes
\Tor(\Z_{m(p,i_p)},\Z_{m(p+1,i_{p+1})}\otimes\dots\otimes\Z_{m(n,i_n)})\\
&\subset T_p. 
\end{align*}
Then one has 
\[
H_0(\G)=\bigoplus_{(i_1,i_2,\dots,i_n)\in J}S(i_1,i_2,\dots,i_n)
\]
and 
\[
T_p=\bigoplus_{(i_1,i_2,\dots,i_n)\in J}T_p(i_1,i_2,\dots,i_n). 
\]
Let $J_0\subset J$ be the subset consisting of $(i_1,i_2,\dots,i_n)$ 
such that $m(d,i_d)\in2\Z$ for all $d$ and 
$\#\{d\mid m(d,i_d)\in4\Z{+}2\}$ is less than three. 
Let $S_0\subset H_0(\G)$ be the subgroup 
\[
S_0=\bigoplus_{(i_1,i_2,\dots,i_n)\in J_0}S(i_1,i_2,\dots,i_n). 
\]
We have 
\begin{align*}
&\Ext\left(\bigoplus_{p=1}^{n-1}T_p,S_0\otimes\Z_2\right)\\
&=\bigoplus_{p=1}^{n-1}\bigoplus_{(i_1,i_2,\dots,i_n)\in J}
\bigoplus_{(i'_1,i'_2,\dots,i'_n)\in J_0}
\Ext\left(T_p(i_1,i_2,\dots,i_n),S(i'_1,i'_2,\dots,i'_n)\otimes\Z_2\right)
\end{align*}
and each summand 
$\Ext(T_p(i_1,\dots,i_n),S(i'_1,\dots,i'_n)\otimes\Z_2)$ is 
either $0$ or $\Z_2$. 
We let 
\[
e(A_1,A_2,\dots,A_n)
\in\Ext\left(\bigoplus_{p=1}^{n-1}T_p,S_0\otimes\Z_2\right)
\]
be the element whose 
$\Ext(T_p(i_1,\dots,i_n),S(i'_1,\dots,i'_n)\otimes\Z_2)$ summand is 
nontrivial if and only if 
\begin{itemize}
\item $i_d{=}i'_d$ for all $d$, 
\item $m(d,i_d)\in4\Z$ for all $d\leq p{-}1$, 
\item $m(p,i_p)\in4\Z+2$, 
\item $\#\{d\mid d>p,\ m(d,i_d)\in4\Z{+}2\}=1$. 
\end{itemize}

Making use of this notation, 
we can describe $[[\G]]_\ab$ in the following way. 

\begin{theorem}\label{prodSFTAbel}
Let $\G=\G_{A_1}\times\G_{A_2}\times\dots\times\G_{A_n}$ be 
a product groupoid of SFT groupoids. 
Let the notation be as above. 
Then 
\[
\begin{CD}
0@>>>S_0\otimes\Z_2@>j>>[[\G]]_\ab@>I>>H_1(\G)@>>>0
\end{CD}
\]
is exact. 
The quotient map $[[\G]]_\ab\to H_1(G)$ has a right inverse 
on the subgroup 
\[
\bigoplus_{q_1+q_2+\dots+q_n=1}
H_{q_1}(\G_{A_1})\otimes H_{q_2}(\G_{A_2})\otimes\dots
\otimes H_{q_n}(\G_{A_n})\subset H_1(\G). 
\]
The extension of the remaining part 
\[
\bigoplus_{p=1}^{n-1}
H_0(\G_{A_1})\otimes\dots\otimes H_0(\G_{A_{p-1}})
\otimes\Tor(H_0(\G_{A_p}),H_0(\G_{A_{p+1}}\times\dots\times \G_{A_n}))
\]
is given by the element $e(A_1,A_2,\dots,A_n)$ described above. 
\end{theorem}
\begin{proof}
Assume 
\[
H_0(\G_{A_d})\cong\bigoplus_{i=1}^{h(d)}\Z_{m(d,i)}
\]
as above. 
Let $a_{d,i}\in\Z_{m(d,i)}\subset H_0(\G_{A_d})$ be a generator. 
Let $Y_{d,i}\subset X_{A_d}$ be a proper nonempty clopen subset 
such that $[1_{Y_{d,i}}]=a_{d,i}$. 

First, we determine the kernel of 
the homomorphism $j:H_0(\G)\otimes\Z_2\to[[\G]]_\ab$. 
Let $(i_1,i_2,\dots,i_n)\in J\setminus J_0$. 
We would like to show that 
$S(i_1,\dots,i_n)\otimes\Z_2$ is contained in the kernel. 
If $m(d,i_d)$ is odd for some $d$, then $S(i_1,\dots,i_n)\otimes\Z_2=0$. 
Suppose that $m(d,i_d)$ is even for all $d$ and 
that $\#\{d\mid m(d,i_d)\in4\Z{+}2\}$ is not less than three. 
To simplify notation, we assume $m(d,i_d)\in4\Z{+}2$ for $d=1,2,3$. 
Let $\phi_d$ be a homomorphism $\phi_d:\Z_2\to H_0(\G_{A_d})$ such that 
\[
\phi_d(\bar{1})=\frac{m(d,i_d)}{2}a_{d,i}\in\Z_{m(d,i_d)}
\]
for $d=1,2,3$. 
Set 
\[
\H=\G_{[3]}\times\G_{[3]}\times\G_{[3]}
\times(\G_{A_4}|Y_{4,i_4})\times\dots\times(\G_{A_n}|Y_{n,i_n})
\]
and 
\[
Y=Y_{1,i_1}\times Y_{2,i_2}\times\dots\times Y_{n,i_n}\subset\G^{(0)}. 
\]
Applying Theorem \ref{embedding} to $\H$ and $\G|Y$, 
we get a unital homomorphism $\pi:C^*_r(\H)\to C^*_r(\G|Y)$. 
Let $\tau\in[[\G_{[3]}\times\G_{[3]}\times\G_{[3]}]]$ be a transposition 
corresponding to the nontrivial element of $\Z_2$. 
By Lemma \ref{tauiszero}, $\tau$ is in the commutator subgroup. 
Consider the transposition $\tau\times\id\in\Homeo(\H^{(0)})$, 
which is in $D([[\H]])$. 
The homomorphism $\pi$ induces an embedding $[[\H]]\to[[\G|Y]]$, 
which maps $\tau\times\id$ to a transposition corresponding to 
$a_{1,i_1}\otimes\dots\otimes a_{n,i_n}\otimes\bar{1}\in H_0(\G)\otimes\Z_2$. 
Therefore 
the image of $a_{1,i_1}\otimes\dots\otimes a_{n,i_n}\otimes\bar{1}$ by $j$ 
is zero in $[[\G]]_\ab$. 

Assume $(i_1,i_2,\dots,i_n)\in J_0$. 
For each $d\in\{1,2,\dots,n\}$, we define a homomorphism $\phi_d$ 
from $H_0(\G_{A_d})$ to a finite cyclic group. 
When $m(d,i_d)$ is in $4\Z$, we define $\phi_d:H_0(\G_{A_d})\to\Z_4$ by 
\[
\phi_d(a_{d,i})=\begin{cases}\bar{1}&i=i_d\\\bar{0}&i\neq i_d. \end{cases}
\]
When $m(d,i_d)$ is in $4\Z{+}2$, 
we define $\phi_d:H_0(\G_{A_d})\to\Z_2$ by the same formula. 
Set 
\[
\H=\G_{[k_1]}\times\G_{[k_2]}\times\dots\times\G_{[k_n]}, 
\]
where 
\[
k_d=\begin{cases}3&m(d,i_d)\in4\Z{+}2\\5&m(d,i_d)\in4\Z. \end{cases}
\]
Applying Theorem \ref{embedding} to $\G|Y$ and $\H$, 
we get a unital homomorphism $\pi:C^*_r(\G|Y)\to C^*_r(\H)$. 
Let $\tau\in[[\H]]$ be a transposition 
corresponding to the nontrivial element of $H_0(\H)\otimes\Z_2\cong\Z_2$. 
Since $\#\{d\mid m(d,i_d)\in4\Z{+}2\}$ is less than three, 
by Lemma \ref{tauisnonzero}, there exists a homomorphism $\rho$ 
from $[[\H]]_\ab$ to a finite cyclic group 
such that $\rho([\tau])$ is nonzero. 
The homomorphism $\pi$ induces an embedding $[[\G|Y]]\to[[\H]]$, 
and hence a homomorphism $\pi_*:[[\G|Y]]_\ab\to[[\H]]_\ab$. 
The composition $\rho\circ\pi_*$ satisfies 
\[
\rho(\pi_*(j(a_{1,l_1}\otimes\dots\otimes a_{n,l_n}\otimes\bar{1})))
\neq\bar{0}
\iff(l_1,l_2,\dots,l_n)=(i_1,i_2,\dots,i_n). 
\]
Hence we can conclude 
\[
\Ker j
=\bigoplus_{(i_1,i_2,\dots,i_n)\notin J_0}S(i_1,i_2,\dots,i_n)\otimes\Z_2, 
\]
which implies that 
\[
\begin{CD}
0@>>>S_0\otimes\Z_2@>j>>[[\G]]_\ab@>I>>H_0(\G)@>>>0
\end{CD}
\]
is exact. 

Next, let us determine the equivalence class of the extension above. 
Let $q_1+q_2+\dots+q_n=1$ and consider the subgroup 
\[
H_{q_1}(\G_{A_1})\otimes H_{q_2}(\G_{A_2})\otimes\dots
\otimes H_{q_n}(\G_{A_n})\subset H_1(\G). 
\]
We would like to show that 
the quotient map $[[\G]]_\ab\to H_1(\G)$ has a right inverse on this subgroup. 
To simplify the notation, 
we may assume $q_1=1$ and $q_2=q_3=\dots=q_n=0$. 
The group 
\[
H_1(\G_{A_1})\otimes H_0(\G_{A_2})\otimes H_0(\G_{A_3})\otimes\dots
\otimes H_0(\G_{A_n})
\]
is isomorphic to the direct sum of $S(i_1,i_2,\dots,i_n)$ 
for $(i_1,i_2,\dots,i_n)$ such that $m(1,i_1)=0$. 
If $m(d,i_d)=0$ for all $d$, then $S(i_1,i_2,\dots,i_n)$ is $\Z$, 
and so the exact sequence clearly splits on it. 
Suppose that $m(d,i_d)>0$ for some $d\geq2$. 
Let $g$ be the greatest common divisor of $\{m(d,i_d)\mid d\geq2\}$. 
Then $S(i_1,i_2,\dots,i_n)$ is isomorphic to $\Z_g$. 
Let $\gamma\in[[\G_{A_1}]]$ be an element 
whose index is equal to $a_{1,i_1}\in\Z_{m(1,i_1)}\subset H_1(\G_{A_1})$. 
Set 
\[
Y=Y_{2,i_2}\times Y_{3,i_3}\times\dots\times Y_{n,i_n}. 
\]
It follows that $\gamma\odot\id_Y\in[[\G]]$ 
(or its equivalence class in the abelianization) is a lift 
of the generator of $S(i_1,i_2,\dots,i_n)$. 
By Lemma \ref{baker} (1) and (2), 
$\gamma\odot\id_Y\in[[\G]]$ is of order $g$ in $[[\G]]_\ab$. 
Therefore, the exact sequence splits on $S(i_1,i_2,\dots,i_n)$. 

Lastly, 
let us consider the subgroup $T_p\subset H_1(\G)$ for $p=1,2,\dots,n{-}1$. 
Take $(i_1,i_2,\dots,i_n)\in J$. 
When $m(p,i_p)=0$ or $m(d,i_d)=0$ for all $d>p$, then 
$T_p(i_1,i_2,\dots,i_n)$ is zero. 
So, we may assume that $m(p,i_p)>0$ and $m(d,i_d)>0$ for some $d>p$. 
Let $g_0$ be the greatest common divisor of $\{m(d,i_d)\mid d>p\}$, 
so that $\Z_{m(p+1,i_{p+1})}\otimes\dots\otimes\Z_{m(n,i_n)}\cong\Z_{g_0}$. 
Let $g=\gcd(m(p,i_p),g_0)$ and let 
\[
a=\frac{m(p,i_p)}{g}a_{p,i_p},\quad 
b=\frac{g_0}{g}(a_{p+1,i_{p+1}}\otimes a_{p+2,i_{p+2}}\otimes\dots
\otimes a_{n,i_n}). 
\]
One has $ga=0$ and $gb=0$. 
Let $\beta\in[[\G_{A_p}\times\G_{A_{p+1}}\times\dots\times\G_{A_n}]]$ be 
the element associated with $g$, $a$ and $b$ 
(see the construction stated before Lemma \ref{baker}). 
Set 
\[
Y=Y_{1,i_1}\times Y_{2,i_2}\times\dots\times Y_{p-1,i_{p-1}}. 
\]
We can verify that $\id_Y\odot\beta\in[[\G]]$ is 
a lift of the generator of $T_p(i_1,i_2,\dots,i_n)$. 
Let $g_1$ be the greatest common divisor of $\{m(d,i_d)\mid d<p\}$ and 
let $g_2=\gcd(g_1,g)$. 
Then $T_p(i_1,i_2,\dots,i_n)$ is isomorphic to $\Z_{g_2}$. 
There exist $u,v\in\Z$ such that $g_2=ug_1+vg$. 
When $g_1>0$, 
$g_1[1_Y]$ is zero in $H_0(\G_{A_1}\times\dots\times\G_{A_{p-1}})$. 
It follows from Lemma \ref{baker} (1) and (2) that 
$g_1[\id_Y\odot\beta]$ is zero in $[[\G]]_\ab$. 
When $g_1$ is zero, $g_1[\id_Y\odot\beta]$ is clearly zero. 
Hence, we have 
\[
g_2[\id_Y\odot\beta]=ug_1[\id_Y\odot\beta]+vg[\id_Y\odot\beta]
=vg[\id_Y\odot\beta]
\]
in $[[\G]]_\ab$. 
This value is zero if and only if 
the extension splits on $T_p(i_1,i_2,\dots,i_n)$. 
By Lemma \ref{baker} (4), we get 
\begin{align*}
g[\id_Y\odot\beta]&=\frac{g(g+1)}{2}j([1_Y]\otimes a\otimes b\otimes\bar{1})\\
&=\frac{g(g+1)}{2}\frac{m(p,i_p)}{g}\frac{g_0}{g}
j(a_{1,i_1}\otimes a_{2,i_2}\otimes\dots\otimes a_{n,i_n}\otimes\bar{1}). 
\end{align*}
We have already shown that 
$j(a_{1,i_1}\otimes a_{2,i_2}\otimes\dots\otimes a_{n,i_n}\otimes\bar{1})
\neq0$ 
if and only if $(i_1,i_2,\dots,i_n)\in J_0$, i.e. 
\begin{itemize}
\item $m(d,i_d)\in2\Z$ for all $d$, 
\item $\#\{d\mid m(d,i_d)\in4\Z{+}2\}$ is less than three. 
\end{itemize}
Assume $(i_1,i_2,\dots,i_n)\in J_0$. 
If $m(p,i_p)\in4\Z$ and $g_0\in4\Z$, then $g$ is also in $4\Z$. 
It follows that $g(g+1)/2$ is even, which implies $g[\id_Y\odot\beta]=0$. 
If $m(p,i_p)\in4\Z$ and $g_0\in4\Z+2$, then $g$ is in $4\Z+2$. 
It follows that $m(p,i_p)/g$ is even, 
which again implies $g[\id_Y\odot\beta]=0$. 
In the same way, if $m(p,i_p)\in4\Z+2$ and $g_0\in4\Z$, 
then $g[\id_Y\odot\beta]=0$. 
Assume $m(p,i_p)\in4\Z+2$ and $g_0\in4\Z+2$. 
Then, $\#\{d\mid d>p,\ m(d,i_d)\in4\Z{+}2\}=1$ and 
$m(d,i_d)\in4\Z$ for every $d<p$. 
Also, $g$ is in $4\Z+2$, which implies 
\begin{align*}
g[\id_Y\odot\beta]&=\frac{g(g+1)}{2}\frac{m(p,i_p)}{g}\frac{g_0}{g}
j(a_{1,i_1}\otimes a_{2,i_2}\otimes\dots\otimes a_{n,i_n}\otimes\bar{1})\\
&=j(a_{1,i_1}\otimes a_{2,i_2}\otimes\dots\otimes a_{n,i_n}\otimes\bar{1}). 
\end{align*}
Moreover, one has $g_1\in4\Z$ and $g_2\in4\Z+2$. 
Therefore $v$ is odd. 
Consequently, we obtain 
\begin{align*}
g_2[\id_Y\odot\beta]=vg[\id_Y\odot\beta]
&=vj(a_{1,i_1}\otimes\dots\otimes a_{n,i_n}\otimes\bar{1})\\
&=j(a_{1,i_1}\otimes\dots\otimes a_{n,i_n}\otimes\bar{1})\neq0, 
\end{align*}
which completes the proof. 
\end{proof}

\begin{corollary}
Let $\G=\G_{A_1}\times\G_{A_2}\times\dots\times\G_{A_n}$ be 
a product groupoid of SFT groupoids. 
\begin{enumerate}
\item If $n=1$ or $n=2$, then $\G$ has the strong AH property. 
\item When $n\geq3$, $\G$ has the strong AH property if and only if 
the number of $d\in\{1,2,\dots,n\}$ such that $H_0(\G_{A_d})$ contains 
$\Z_2$ as a direct summand is less than three. 
\end{enumerate}
\end{corollary}


\begin{thebibliography}{99}
\bibitem{BBG06CMP}
J. Bellissard, R. Benedetti, and J.-M. Gambaudo, 
\textit{Spaces of tilings, finite telescopic approximations 
and gap-labeling}, 
Comm. Math. Phys. 261 (2006), 1--41. 
\bibitem{BO03}
M. Benameur and H. Oyono-Oyono, 
\textit{Gap-labelling for quasi-crystals 
(proving a conjecture by J. Bellissard)}, 
Operator algebras and mathematical physics (Constanta, 2001), 11--22, 
Theta, Bucharest, 2003. 
\bibitem{Btext}
B. Blackadar, 
\textit{$K$-theory for operator algebras}, 
Mathematical Sciences Research Institute Publications, 5. 
Springer-Verlag, New York, 1986. 
\bibitem{BL10GeomD}
C. Bleak and D. Lanoue, 
\textit{A family of non-isomorphism results}, 
Geom. Dedicata 146 (2010), 21--26. 
arXiv:0807.4955
\bibitem{Brin04GeomD}
M. G. Brin, 
\textit{Higher dimensional Thompson groups}, 
Geom. Dedicata 108 (2004), 163--192. 
arXiv:math/0406046
\bibitem{Brin05JAlg}
M. G. Brin, 
\textit{Presentations of higher dimensional Thompson groups}, 
J. Algebra 284 (2005), 520--558. 
arXiv:math/0501082
\bibitem{Brin10PublMat}
M. G. Brin, 
\textit{On the baker's map and 
the simplicity of the higher dimensional Thompson groups $nV$}, 
Publ. Mat. 54 (2010), 433--439. 
arXiv:0904.2624
\bibitem{Br87JPAA}
K. S. Brown, 
\textit{Finiteness properties of groups}, 
J. Pure Appl. Algebra 44 (1987), 45--75. 
\bibitem{CM00crelle}
M. Crainic and I. Moerdijk, 
\textit{A homology theory for \'etale groupoids}, 
J. Reine Angew. Math. 521 (2000), 25--46. 
\bibitem{CK80Invent}
J. Cuntz and W. Krieger, 
\textit{A class of $C^*$-algebras and topological Markov chains}, 
Invent. Math. 56 (1980), 251--268. 
\bibitem{DMP14Israel}
W. Dicks and C. Mart\'inez-P\'erez, 
\textit{Isomorphisms of Brin-Higman-Thompson groups}, 
Israel J. Math. 199 (2014), 189--218. 
arXiv:1112.1606
\bibitem{FMWZ13PJM}
M. Fluch, M. Marschler, S. Witzel and M. C. B. Zaremsky, 
\textit{The Brin-Thompson groups $sV$ are of type F$_\infty$}, 
Pacific J. Math. 266 (2013), 283--295. 
arXiv:1207.4832
\bibitem{FH99ETDS}
A. H. Forrest and J. R. Hunton, 
\textit{The cohomology and $K$-theory of commuting homeomorphisms 
of the Cantor set}, 
Ergodic Theory Dynam. Systems 19 (1999), 611--625. 
\bibitem{GPS95crelle}
T. Giordano, I. F. Putnam and C. F. Skau, 
\textit{Topological orbit equivalence and $C^*$-crossed products}, 
J. Reine  Angew. Math. 469 (1995), 51--111. 
\bibitem{GPS99Israel}
T. Giordano, I. F. Putnam and C. F. Skau, 
\textit{Full groups of Cantor minimal systems}, 
Israel J. Math. 111 (1999), 285--320. 
\bibitem{GPS04ETDS}
T. Giordano, I. F. Putnam and C. F. Skau, 
\textit{Affable equivalence relations and orbit structure 
of Cantor dynamical systems}, 
Ergodic Theory Dynam. Systems 24 (2004), 441--475. 
\bibitem{GJtext}
P. G. Goerss and J. F. Jardine, 
\textit{Simplicial homotopy theory}, 
Progress in Mathematics, 174. Birkh\"auser Verlag, Basel, 1999. 
\bibitem{HS84AnnInst}
P. de la Harpe and G. Skandalis, 
\textit{D\'eterminant associ\'e \`a une trace sur une alg\'ebre de Banach}, 
Ann. Inst. Fourier (Grenoble) 34 (1984), 241--260. 
\bibitem{HM12PJM}
J. Hennig and F. Matucci
\textit{Presentations for the higher-dimensional Thompson groups $nV$}, 
Pacific J. Math. 257 (2012), 53--74. 
arXiv:1105.3714
\bibitem{JM13Ann}
K. Juschenko and N. Monod, 
\textit{Cantor systems, piecewise translations and simple amenable groups}, 
Ann. of Math. 178 (2013), 775--787. 
arXiv:1204.2132
\bibitem{KP03Michigan}
J. Kaminker and I. F. Putnam, 
\textit{A proof of the gap labeling conjecture}, 
Michigan Math. J. 51 (2003), 537--546. 
%\bibitem{Kato15}
%M. Kato, 
%\textit{Higher dimensional Thompson groups have subgroups 
%with infinitely many relative ends}, 
%preprint. 
%arXiv:1504.06680
\bibitem{KMPN13ProcEdinb}
D. H. Kochloukova, C. Mart\'inez-P\'erez and B. E. A. Nucinkis, 
\textit{Cohomological finiteness properties of the Brin-Thompson-Higman groups 
$2V$ and $3V$}, 
Proc. Edinb. Math. Soc. (2) 56 (2013), 777--804. 
arXiv:1009.4600
\bibitem{LM}
D. Lind and B. Marcus, 
\textit{An introduction to symbolic dynamics and coding}, 
Cambridge University Press, Cambridge, 1995. 
\bibitem{MPMN13}
C. Mart\'inez-P\'erez, F. Matucci and B. E. A. Nucinkis, 
\textit{Cohomological finiteness conditions and 
centralisers in generalisations of Thompson's group $V$}, 
arXiv:1309.7858 
\bibitem{MM14Kyoto}
K. Matsumoto and H. Matui, 
\textit{Continuous orbit equivalence of topological Markov shifts 
and Cuntz-Krieger algebras}, 
Kyoto J. Math. 54 (2014), 863--877. 
arXiv:1307.1299
\bibitem{M06IJM}
H. Matui, 
\textit{Some remarks on topological full groups of Cantor minimal systems}, 
Internat. J. Math. 17 (2006), 231--251. 
arXiv:math/0404117
\bibitem{M12PLMS}
H. Matui, 
\textit{Homology and topological full groups of \'etale groupoids 
on totally disconnected spaces}, 
Proc. Lond. Math. Soc. (3) 104 (2012), 27--56. 
arXiv:0909.1624
\bibitem{M15crelle}
H. Matui, 
\textit{Topological full groups of one-sided shifts of finite type}, 
J. Reine Angew. Math. 705 (2015), 35--84. 
arXiv:1210.5800
\bibitem{Ne04JOP}
V. V. Nekrashevych, 
\textit{Cuntz-Pimsner algebras of group actions}, 
J. Operator Theory 52 (2004), 223--249. 
\bibitem{Ne13}
V. V. Nekrashevych, 
\textit{Finitely presented groups associated with expanding maps}, 
preprint. 
arXiv:1312.5654
\bibitem{Ne15}
V. V. Nekrashevych, 
\textit{Simple groups of dynamical origin}, 
preprint. 
arXiv:1511.08241
\bibitem{P90ETDS}
I. F. Putnam, 
\textit{On the topological stable rank of 
certain transformation group $C^*$-algebras}, 
Ergodic Theory Dynam. Systems 10 (1990), 197--207. 
\bibitem{Rtext}
J. Renault, 
\textit{A groupoid approach to $C^*$-algebras}, 
Lecture Notes in Mathematics 793, Springer, Berlin, 1980. 
\bibitem{R08Irish}
J. Renault, 
\textit{Cartan subalgebras in $C^*$-algebras}, 
Irish Math. Soc. Bull. 61 (2008), 29--63. 
arXiv:0803.2284
\end{thebibliography}
\end{document}